\documentclass[11pt]{amsart}

\usepackage[marginratio=1:1,height=660pt,width=480pt,tmargin=70pt]{geometry}

\usepackage{amssymb,latexsym,amsmath,amsthm,amscd}
\usepackage{mathrsfs}   
\usepackage{fancyhdr}
 \usepackage{ifthen} 
 \usepackage{xstring}

\usepackage{etoolbox}

\newcommand{\im}{\lrcorner\,}

\appto\appendix{\addtocontents{toc}{\protect\setcounter{tocdepth}{1}}}
\appto\listoffigures{\addtocontents{lof}{\protect\setcounter{tocdepth}{1}}}
\appto\listoftables{\addtocontents{lot}{\protect\setcounter{tocdepth}{1}}}
\theoremstyle{plain}
\newtheorem{theorem}{Theorem}[section]
\newtheorem{corollary}[theorem]{Corollary}

\newtheorem{lemma}[theorem]{Lemma}
\newtheorem{proposition}[theorem]{Proposition}
\newtheorem{definition}[theorem]{Definition}

\theoremstyle{remark}

\newtheorem{remark}[theorem]{Remark}

\newtheorem*{example*}{Example}

\numberwithin{equation}{section}

\usepackage{enumitem,array,rotating}  
\usepackage{color}
\usepackage{xcolor}
\definecolor{carmine}{rgb}{0.59, 0.0, 0.09}
\definecolor{mediumpersianblue}{rgb}{0.0, 0.4, 0.65}
\definecolor{persianplum}{rgb}{0.44, 0.11, 0.11}
\usepackage[colorlinks=true,
            linkcolor=persianplum, 
            urlcolor=olive,
            citecolor=mediumpersianblue, 
            backref=page]{hyperref}
\usepackage{amsfonts}
  
\newcommand{\sN}{\mathsf{N}}
\newcommand{\sQ}{\mathsf{Q}}
\newcommand{\sP}{\mathsf{P}}

\newcommand{\cA}{\mathcal{A}}
\newcommand{\cB}{\mathcal{B}}
\newcommand{\cC}{\mathcal{C}}

\newcommand{\cG}{\mathcal{G}}
\newcommand{\cH}{\mathcal{H}}
\newcommand{\cL}{\mathcal{L}}

\newcommand{\cS}{\mathcal{S}}
\newcommand{\cP}{\mathcal{P}}
\newcommand{\cF}{\mathcal{F}}
\newcommand{\cJ}{\mathcal{J}}

\newcommand{\cT}{\mathcal{T}^{[g]}}
\newcommand{\cTT}{\mathcal{T}^{[\nabla]}}
\newcommand{\cTTg}{\mathcal{T}^{[\nabla^{[g]}]}}
\newcommand{\cQ}{\mathcal{Q}}

\newcommand{\cI}{\mathcal{I}}
\newcommand{\cN}{\mathcal{N}}
\newcommand{\cK}{\mathcal{K}}
\newcommand{\bcL}{\overline{\mathcal{L}}}

\newcommand{\bvarphi}{\Psi}
\newcommand{\bmu}{\Gamma}
\newcommand{\beeta}{\Lambda}
\newcommand{\btau}{\Omega}

\newcommand{\scV}{\mathscr{V}}
\newcommand{\scU}{\mathscr{U}}
\newcommand{\scD}{\mathscr{D}}
\newcommand{\scE}{\mathscr{E}}
\newcommand{\scF}{\mathscr{F}}
\newcommand{\scK}{\mathscr{K}}

\newcommand{\scC}{\mathscr{C}}
\newcommand{\scH}{\mathscr{H}}
\newcommand{\scT}{\mathscr{T}}

\newcommand{\DD}{\mathbb{D}}

\newcommand{\fso}{\mathfrak{so}}

\newcommand{\fg}{\mathfrak{g}}
\newcommand{\fp}{\mathfrak{p}}

\newcommand{\ri}{\mathrm{i}}

\newcommand{\RR}{\mathbb{R}}
\newcommand{\KK}{\mathbb{K}}

\newcommand{\CC}{\mathbb{C}}

\newcommand{\PP}{\mathbb{P}}

\newcommand{\FF}{\mathbb{F}}
\newcommand{\SSS}{\mathbb{S}}

\newcommand{\RT}{\mathcal{R}^{\uparrow}T^*M}

\newcommand{\bb}{\mathbf{b}}
\newcommand{\bc}{\mathbf{c}}

\newcommand{\bg}{\mathbf{g}}

\newcommand{\bq}{\mathbf{q}}

\newcommand{\calpha}{\overline{\alpha}}
\newcommand{\cbeta}{\overline{\beta}}
\newcommand{\cphi}{\overline{\phi}}
\newcommand{\csigma}{\overline{\sigma}}
\newcommand{\cpi}{\overline{\pi}}
\newcommand{\comega}{\overline{\omega}}
\newcommand{\ctheta}{\overline{\theta}}
\newcommand{\cxi}{\overline{\xi}}
\newcommand{\czeta}{\overline{\zeta}}
\newcommand{\cscH}{\overline{\mathscr{H}}}
\newcommand{\cscK}{\overline{\mathscr{K}}}

\newcommand{\w}{{\,{\wedge}\;}}

\newcommand{\exd}{\mathrm{d}}

\newcommand{\ve}{\varepsilon}

\newcommand{\Id}{\operatorname{Id}}

\newcommand{\half}{\textstyle{\frac 12}}
 
\newcommand{\Ker}{\mathrm{Ker}}

\newcommand{\biw}{\bigwedge\nolimits}
\newcommand{\what}{\widehat}

\makeatletter
\renewcommand*{\p@section}{\S\,}
\renewcommand*{\p@subsection}{\S\,}
\renewcommand*{\p@subsubsection}{\S\,}

\makeatother
\usepackage{float}

\begin{document}

\author{Omid Makhmali}

\address{\newline    Omid Makhmali\\\newline
   Departamento de Geometr\'ia y Topolog\'ia and IMAG, Universidad de Granada, Granada 18071, Spain \\\newline
    \textit{Email address: }{\href{mailto:omakhmali@ugr.es}{\texttt{omakhmali@ugr.es}}}\\\newline
Department of Mathematics and Statistics, UiT The Arctic University of Norway, Troms\o\  90-37,Norway \\\newline
\textit{Email address: }{\href{mailto:omid.makhmali@uit.no}{\texttt{omid.makhmali@uit.no}}}\\\newline
Department of Mathematics and Natural Sciences, Cardinal Stefan Wyszy\'nski University, ul. Dewajtis 5,  Warszawa, 01-815, Poland \\\newline
\textit{Email address: }{\href{mailto:o.makhmali@uksw.edu.pl}{\texttt{o.makhmali@uksw.edu.pl}}}
}

\title[]
{Weyl metrizability of 3-dimensional projective structures and CR submanifolds} 
\date{\today}

\begin{abstract}
  A projective structure is Weyl metrizable if it has a representative that preserves a conformal structure. We interpret Weyl metrizability of
  3-dimensional projective structures as certain 5-dimensional nondegenerate CR submanifolds in a class of 7-dimensional 2-nondegenerate CR structures. As a corollary, it follows  that in dimension three Beltrami's theorem extends to conformal structures, i.e.  a locally flat projective structure is Weyl metrizable exclusively with respect to a locally flat conformal structure. In higher dimensions it is  shown that conformal Beltrami theorem remains true as well.

\end{abstract}


\maketitle
  
\vspace{-.5 cm}

\setcounter{tocdepth}{1}   
\tableofcontents

\section{Introduction}
\label{sec:introduction}
A projective structure on a manifold $M,$ denoted as $[\nabla],$ is defined as the equivalence class of  torsion-free linear connections whose geodesics coincide as unparametrized curves. The metrizability problem for $[\nabla]$, which goes back to R. Liouville \cite{Roger}, asks whether  $[\nabla]$ admits a representative $\nabla\in[\nabla]$ that is (locally) the Levi-Civita connection of some (pseudo-)Riemannian metric on $M.$   A slightly weaker question,  referred to as  the (local) Weyl metrizability problem, is whether   (locally) there is a  (pseudo-)Riemannian conformal structure $[g]$ on $M$ that is preserved by a representative $\nabla\in[\nabla],$ i.e. $\nabla_Xg\in [g]$ for some, and therefore all,  $g\in [g]$ along any vector field $X$ on $M.$

The metrizability problem on surfaces has been intensively studied and a derivation of necessary and sufficient conditions for a  2-dimensional projective structure to be metrizable can be found in \cite{BDE,EM}.  Metrizability of  3-dimensional projective structures has been studied in \cite{DE-3D,E-3D} and some necessary conditions were derived. For more recent developments on the metrizability problem we refer to  \cite{Flood}. The Weyl metrizability problem of 2-dimensional projective structures  was shown in \cite{Mettler-1} to be locally unobstructed. 

To highlight some local aspects of these questions, we first recall from \cite{EM} that generic projective structures are not (locally) metrizable. It will be instructive to explain this fact by counting the local generality of projective structures and (pseudo-)Riemannian metrics. Recall that on an $n$-dimensional manifold $M$, a torsion-fee linear connection $\nabla$ is locally defined by its  Christoffel symbols, $\Gamma^i_{jk}=\Gamma^i_{kj},$ which  are $\half{n^2(n+1)}$ functions of $n$-variables. However, as we will see in \eqref{eq:proj-equiv}, the projective class $[\nabla]$ is determined by the torsion-free part of $\Gamma^i_{jk}$'s. Thus,   up to diffeomorphisms,  projective structures locally depend on   $a_1=\half{n^2(n+1)}-2n$ functions of $n$ variables.   The metrizability problem can be expressed as  an overdetermined first order  PDE system $\Gamma^i_{jk}=\Gamma^i_{jk}(g)$ for a (pseudo-)Riemannian metric $g$ as its solution.  Note that, up to diffeomorphisms, (pseudo-)Riemannian metrics on $M$ are determined by  $a_2=\half{n(n+1)}-n=\half{n(n-1)}$ functions of $n$ variables. Thus, naively speaking, one expects to find at least $a_1-a_2=\half{n(n^2-3)}$ obstructions for the (local) metrizability of a generic projective structure on $M,$ which explains the non-metrizability of generic projective structures.

Regarding the question of Weyl metrizability, we first find the local generality of Weyl structures by noting that from the definition above $[\nabla]$ is Weyl metrizable if there is a connection $\nabla\in[\nabla],$ a 1-form  $\beta,$ and a metric $g$ such that $\nabla g=2\beta\otimes g.$ It is immediate to check that  $(\beta,g)$ and $(\beta+\exd\lambda,e^{2\lambda}g)$ encode Weyl metrizability of $\nabla$ with respect to the same conformal structure.   Thus, Weyl metrizable projective structures are encoded by a 1-form $\beta$ and a conformal structure $[g],$ i.e. their local generality, up to diffeomorphisms,  is $a_3=n+\half{n(n-1)}-1$ functions of $n$ variables.   As a result, naively speaking, one expects to find at least $a_1-a_3=\half{(n-2)(n^2+2n-1)}$ obstructions for the (local) Weyl metrizability of a generic projective structure on $M.$

The naive count above suggests that all projective structures on surfaces are locally Weyl metrizable, which was shown in \cite{Mettler-1} by exploiting the twistor bundle of projective structures on surfaces, as a 2-disk bundle with a tautologically induced complex structure. The result follows from the  one-to-one correspondence between  (local) holomorphic sections of the twistor bundle of $[\nabla]$  and (local) Weyl structures for $[\nabla]$.

In dimension three our naive count suggests the existence of local obstructions for Weyl metrizability. In this article the (local) existence of a Weyl structure for a  3-dimensional projective structure is shown to be equivalent to the (local) existence of certain 5-dimensional Levi nondegenerate CR submanifolds in a special class of 7-dimensional Levi 2-nondegenerate CR manifolds. The Levi nondegenerate CR submanifolds are obtained from the so-called  twistor CR manifolds which are the twistor bundle of 3-dimensional conformal  structures and have been studied in the twistor programme \cite{LeBrun-CR,MasonPhD}. The class of Levi 2-nondegenerate CR manifolds are the twistor bundle of 3-dimensional projective structures and have been studied recently \cite{Gregorovic-CR,KM-subconformal}. 
 
As a corollary, we see that in dimension three flat projective structures are exclusively Weyl metrizable with respect to flat conformal structures. We will show that this fact remains true in higher dimensions as well. Lastly, we give a CR characterization of the Einstein-Weyl condition as the descent of the CR complex structure to the 4-dimensional space of  (time-like) geodesics of the Weyl structure.

\subsection{Outline of the article and main results}
\label{sec:outline-article-main}
Throughout this article the Cartan geometric language will be employed to significantly simplify the constructions involved. We start  \ref{sec:from-3d-conformal} with a brief definition of a Cartan geometry and the notion of Weyl structures for a special class of Cartan geometries referred to as  parabolic geometries. Throughout the article, when dealing with Weyl structures for CR, projective, and conformal structures we will sometimes use the terms CR, projective, conformal connections, respectively. As a result, the Weyl structure that was described previously will also be called a conformal connection for some $[g].$  The rest of \ref{sec:from-3d-conformal} briefly recalls the treatment of 3-dimensional conformal structures of Riemannian and Lorentzian signature, denoted by $[g],$ as certain parabolic geometries and reviews the construction of their twistor bundle, also known as  twistor CR manifolds, denoted by $\cT$. The main observation in this section is Proposition \ref{prop:algebr-type-CR-twist} where it is shown that the holomorphic/fundamental binary quartic of twistor CR manifolds, expressed in \eqref{eq:W-CR-quartic},  is either zero or has a repeated root of multiplicity four. In Remark \ref{rmk:petrov-type-twistor} we comment on another instance of such results in the context of conformal Fefferman metrics for 3-dimensional CR structures.

 The Cartan geometric treatment of a 3-dimensional (oriented) projective structure, $[\nabla],$ and the construction of its twistor bundle, denoted by $\cTT,$ is briefly reviewed in \ref{sec:from-3d-projective}.  In particular, we show that the 7-dimensional twistor bundle, $\cTT,$ carries a  Levi 2-nondegenerate CR structure of hypersurface type which is canonically determined by the underlying oriented 3-dimensional projective structure. See Definition \ref{def:cr-para-cr} for the notion of 2-nondegeneracy.

 The CR characterization of Weyl metrizability is carried out in \ref{sec:weyl-metr-as}, leading to the following theorem.
 \theoremstyle{plain}
\newtheorem*{thmA}{\bf Theorem \ref{thm:characterization-weyl-metr}}
\begin{thmA}
    Given an oriented projective structure $[\nabla]$ on a 3-manifold $M,$   there is a one-to-one correspondence between conformal structures $[g]$ with respect to which $[\nabla]$ is Weyl metrizable and 5-dimensional Levi nondegenerate CR submanifolds of hypersurface type that are transverse to the Levi kernel of the 7-dimensional twistor bundle, $\cTT,$  whose fundamental binary quartic is either zero or has a repeated root of multiplicity  four   and satisfy the non-integrability condition \eqref{eq:nondege-conformal} for $\ve= 1$ or $-1$ with respect to the canonical CR connection described in Proposition \ref{prop:Weyl-structre-cr-submanifolds}.  
  \end{thmA}
  In the theorem above  $\ve$ determines whether $[g]$ is Lorentzian or Riemannian. In Corollary \ref{cor:CR-type-N} we study all Levi nondegenerate CR submanifolds of $\cTT$ that are transverse to its Levi kernel whose holomorphic binary quartic is either zero or has a repeated root of multiplicity four.

  In \ref{sec:two-corollaries} we highlight two corollaries arising from \ref{sec:weyl-metr-as}.  First corollary shows projectively flat Weyl structures are rigid in dimension three, which we extend to higher dimensions in the following theorem. 
 \theoremstyle{plain}
\newtheorem*{thmB}{\bf Theorem \ref{thm:conf-beltrami}}
\begin{thmB}
  A  conformal structure in dimension $\geq 3$ locally has a projectively flat Weyl structure if and only if it is locally conformally flat. 
\end{thmB}
Recall that the classical Beltrami theorem  \cite{Beltrami} says that projectively flat (pseudo)-Riemannian metrics have constant sectional curvature. As a result, the theorem above can be viewed as a conformal extension of   Beltrami's theorem. In Remark \ref{rmk:equivalent-description}  an equivalent description of this theorem is given and in Remark \ref{rmk:2D-weyl} we  contrast it with  projectively flat  Weyl structures on surfaces. 

Lastly, we give a CR characterization of the Einstein-Weyl condition and show that a Weyl structure $\nabla^{[g]}$ is Einstein-Weyl if and only if the CR complex structure on $\cT$ descends to the 4-dimensional space of (time-like) geodesics, also referred to as  the minitwistor space of the Einstein-Weyl structure  in \cite{Hitchin-EW}.

\subsection{Conventions}
\label{sec:conventions}

In this article we will work locally over smooth  manifolds. Given a distribution $\scD\subset TM,$  its derived system is the distribution whose sheaf of sections is given by $\Gamma(\scD)+[\Gamma(\scD),\Gamma(\scD)]$ and, by abuse of notation, is denoted as $[\scD,\scD].$ Similarly, given two distributions $\scD_1$ and $\scD_2,$ we denote by $[\scD_1,\scD_2]$ the distribution whose sheaf of sections is   $\Gamma(\scD_1)+\Gamma(\scD_2)+[\Gamma(\scD_1),\Gamma(\scD_1)]+[\Gamma(\scD_1),\Gamma(\scD_2)]+[\Gamma(\scD_2),\Gamma(\scD_2)].$  We denote by $\Omega^k(M)$ and $\Omega^k(M,\fg)$ the sheaf of sections of  $\biw^k T^*M$ and $\biw^k T^*M\otimes\fg,$ respectively.   The symmetric product of  1-forms is denoted by their multiplication, e.g. for two 1-forms $\alpha$ and $\beta$ we define $\alpha\beta=\half(\alpha\otimes\beta+\beta\otimes\alpha)$ and $\alpha^k$ denotes the $k$th symmetric power of $\alpha.$ The span of vector fields  $v_1,\dots,v_k\in\Gamma(T M)$  is denoted by $\langle v_1,\dots,v_k\rangle$ and the algebraic ideal generated by 1-forms $\omega^0,\dots,\omega^n\in\Omega^1(M)$ is denoted as $\{\omega^0,\dots,\omega^n\}.$

When dealing with a Lorentzian conformal structures on $M$,  the bundle of rays of  future time-like   1-forms is denoted by $\RT$. In the case of  conformal structures of Riemannian signature, $\RT$ coincides with the oriented projectivized cotangent bundle   $\PP_+T^*M:=T^*M\slash\RR^+,$ which is the double cover of $\PP T^*M,$   

  Given a principal bundle $\mu\colon\cC\to M,$ let $(\omega^i,\omega^i_j)$ be a coframe on $\cC$ such that $\omega^i$'s are semi-basic with respect to  $\mu\colon\cC\to M,$ i.e. $\omega^i(v)=0$ for all $v\in\Ker\mu_*.$ We denote the frame dual to this coframe by $(\partial_{\omega^i},\partial_{\omega^i_j}).$ Moreover, given  a  function $f$ on $\cC,$ its \emph{coframe derivatives}  are defined by
\[f_{;i}=\partial_{\omega^i}\im\exd f.\]
Except for Theorem \ref{thm:conf-beltrami}, it is assumed that $0\leq i,j,k,l\leq 2$ and    $[\ve_{ij}]=\mathrm{diag}(1,\ve,\ve),$ where  $\ve$ is either $1$ or $-1$ depending on the signature of the conformal structure: It is 1 for the Riemannian case and $-1$ for the Lorentzian case. We use the summation convention over repeated indices and denote symmetrization and skew-symmetrization by round and square brackets around affected indices, respectively, e.g. $\alpha_{(ij)k}=\half (\alpha_{ijk}+\alpha_{jik})$ and $\alpha_{[ij]k}=\half(\alpha_{ijk}-\alpha_{jik}).$ 
Lastly, given a complex-valued     function (or 1-form)  $f=f_1+\ri f_2$ its  real and imaginary parts are denoted by $\Re f=f_1$ and $\Im f=f_2.$

\section{Conformal structures and nondegenerate CR geometry} 
\label{sec:from-3d-conformal}
In this section we recall the notion of Cartan geometries and Weyl structures for the class of parabolic geometries. We review the Cartan geometric description of conformal structures in dimension three and the construction of their twistor bundle, also referred to as twistor CR manifolds. In \ref{sec:cr-structures-type-begin} we review the definition of  CR and para-CR structures of hypersurface type whose Levi bracket is nondegenerate or 2-nondegenerate. We finish this section with Proposition \ref{prop:algebr-type-CR-twist} where  a necessary condition is found for a 5-dimensional nondegenerate CR structure to arise as a  twistor CR manifold in terms of the root type of its holomorphic/fundamental binary quartic. 

\subsection{Parabolic geometries and Weyl structures}\label{sec:parab-goem-weyl}
We start by the definition of  a Cartan geometry and its Cartan curvature.

\begin{definition}
  Let $G$ be a Lie group and $P\subset G$ a Lie subgroup  with Lie algebras $\fg$ and $\fp\subset\fg,$ respectively.  A Cartan geometry of type $(G,P)$ on a manifold $N,$ denoted as $(\cG\to N,\psi),$ is a right principal $P$-bundle $\cG\to N$ equipped with a Cartan connection $\psi\in\Omega^1(\cG,\fg),$ i.e. a $\fg$-valued 1-form on $\cG$ satisfying
  \begin{enumerate}
  \item $\psi$ is $P$-equivariant, i.e. $R_g^*\psi=\mathrm{Ad}_{g^{-1}}\psi$ for all $g\in P,$ where $R_g$ denotes the right action by $g$.
  \item  $\psi_z\colon T_z\cG\to \fg$ is  a linear isomorphism for all $z\in \cG.$
  \item $\psi$ maps fundamental vector fields to their generators, i.e. $\psi(\zeta_X)=X$ for any $X\in\fp$ where $\zeta_X(z):=\frac{\exd}{\exd t}\,\vline_{\,t=0}R_{exp(tX)}(z).$
  \end{enumerate} 
 The 2-form $\Psi\in\Omega^2(\cG,\fg)$ defined as
    \[\Psi(X,Y)=\exd\psi(X,Y)+[\psi(X),\psi(Y)]\quad \text{for\ \ }X,Y\in \Gamma(T\cG),\]
is called the Cartan curvature and is $P$-equivariant and semi-basic with respect to the fibration $\cG\to N.$
\end{definition} 
A Cartan geometry of type $(G,P)$ where $G$ is (real or complex) semi-simple and $P\subset G$ is a parabolic subgroup is referred to as a \emph{parabolic geometry}. It is well-known that given a parabolic subgroup $P\subset G,$ the Lie algebras  $\fg$ and $\fp$ are equipped with gradings
\[\fg=\fg_{-k}\oplus\cdots\oplus\fg_{k},\quad \fp=\fp_0\oplus\fp_{+},\quad \fp_+=\fg_1\oplus\cdots\oplus\fg_{k},\]
where $\fp_+,\fp_0\subset\fp$ are the nilradical and a reductive Levi factor of  $\fp=\fp_0\ltimes\fp_+,$ respectively.
Let the Lie subgroups $P_0,P_+\subset P$ be the connected subgroup whose Lie algebras are $\fp_0,\fp_+,$ respectively.  
\begin{definition}\label{def:weyl-str}
  A (local)  Weyl structure for the parabolic geometry $(\cG\to M,\psi)$ is a (local) $P_0$-equivariant section $s\colon\cG_0\to\cG$ of the principal $P_+$-bundle $\cG\to\cG_0$  where $\cG_0\to M$ is the principal $P_0$-bundle defined as $\cG_0=\cG\slash P_+:=\cG\times_{P_+}P.$
\end{definition}
Given a Weyl structure $s\colon\cG_0\to\cG,$ let $\psi_0$ be the $\fg_0$-valued part of $s^*\psi.$ It is clear that, given a Weyl structure, $\psi_0$ defines a linear connection  on the tangent bundle for the underlying  \emph{$G$-structure} where  $G=P_0\subset \mathrm{GL}(n)$, which is referred to as a \emph{Weyl connection.} We refer the reader to \cite[Section 5]{CS-Parabolic} for more on Weyl structures.

In this article we will mainly focus on  local aspects of certain parabolic geometries, and, therefore, when mentioning a Weyl structure it is understood that we mean a local Weyl structure. Furthermore, since we will encounter Weyl structures for conformal, projective and CR geometries, we will sometimes refer to them as conformal, projective, and CR connections, respectively.

\subsection{Conformal geometry in dimension 3}\label{sec:conf-geom-dimens}
In this section we review the Cartan geometric description of conformal structures in dimension three.
A conformal  structure of signature $(p,q)$ on an $n$-dimensional connected manifold $M,$ where $n=p+q,$  is defined as the equivalence class of (pseudo-)Riemannian metrics of signature $(p,q)$ that are proportional via an everywhere positive function on $M,$ i.e. a  sub-bundle $[g]\subset \mathrm{Sym}^2(T^*M)$ defined as
\[[g]=\left\{ \lambda^2 g\ \vline\ \lambda\in C^{\infty}(M,\RR\setminus \{0\}) \right\},\]
  for some (pseudo-)Riemannian metric $g.$

  Conformal structures of signature $(p,q)$ define Cartan geometries  $(\mu\colon\cC\to M,\psi)$ of type $(\mathrm{SO}(p+1,q+1),Q_1)$ where $Q_1\subset \mathrm{SO}(p+1,q+1)$ is the parabolic subgroup that stabilizes a null line in $\RR^{p+1,q+1}.$ With a choice of orientation, which will be important in this paper, one can define oriented conformal structures as Cartan geometries of type $(\mathrm{SO}^o(p+1,q+1),Q_1)$ where $\mathrm{SO}^o(p+1,q+1)$ is the identity components of $\mathrm{SO}(p+1,q+1)$ and $Q_1$ stabilizes a null ray in $\RR^{p+1,q+1}.$  %
  As  parabolic geometries, they are among the so-called \emph{regular} and \emph{normal} parabolic geometries, which we will not elaborate on. 
The flat model for  conformal structures is the $n$-dimensional  hyperquadric $Q_{p,q}\subset\PP^{n+1}$ defined by the standard inner product of signature $(p+1,q+1)$ on $\RR^{p+1,q+1}.$ 

In dimension three,  the Cartan connection $\psi\in\Omega^1(\cC,\fso(p+1,q+1)),p+q=3,$ can be written as
\begin{equation}\label{eq:conf-cartan-conn}
      \psi=\begin{pmatrix}   
  -\theta^0 &\xi_2 & \xi_1 & \xi_0 & 0\\
  \omega^2 & 0 & -\theta^3 & -\ve\theta^2 & \ve\xi_2\\
    \omega^1 & \theta^3 & 0 & -\ve\theta^1 & \ve\xi_1\\
    \omega^0 & \theta^2 & \theta^1 & 0 & \xi_0\\
    0 & \ve\omega^2 & \ve\omega^1 & \omega^0 &   \theta^0
  \end{pmatrix},
  \end{equation}
where we are using the fact that the conformal geometry of metrics of signature $(p,q)$ and $(q,p)$ are equivalent since the Lie groups $\mathrm{SO}(p+1,q+1)$ and $\mathrm{SO}(q+1,p+1)$ are isomorphic.  The  symmetric bilinear form 
\begin{equation}\label{eq:conformal-metric}
  g=(\omega^0)^2+\ve(\omega^1)^2+\ve(\omega^2)^2\in\Gamma(\mathrm{Sym}^2(T^*\cC))
  \end{equation}
is well-defined and the conformal structure on $M$ is given by  $[s^* g]\subset \mathrm{Sym}^2(T^*M)$ for any section  $s\colon M\to \cC.$ When $\ve$ is $1$ or $-1,$ the conformal structure is Riemannian or Lorentzian, respectively.

The Cartan connection $\psi$ is $\fso(p+1,q+1)$-valued  with respect to the inner product
\[u\cdot u=-2u^1u^5+\ve (u^2)^2+\ve (u^3)^2+(u^4)^2,\quad u=(u^1,\ldots,u^5)\in\RR^5.\]
Its Cartan curvature  is 
\begin{equation}\label{eq:Cartan-curvature-conformal-conn}
 \exd\psi+\psi\w\psi=\begin{pmatrix}
    0 &  \Xi_2 & \Xi_1 & \Xi_0 & 0\\
  0 & 0 & 0  & 0 & \ve\Xi_2\\
  0 & 0 &      0  & 0 & \ve\Xi_1\\
  0 & 0 &  0 & 0 & \Xi_0\\
  0 & 0 & 0 & 0 &  0 
\end{pmatrix}
\end{equation}
where 
\begin{equation}\label{eq:cotton-york}
  \begin{aligned}
  \Xi_0=&Y_{001}\omega^0\w\omega^1+Y_{002}\omega^0\w\omega^2+Y_{012}\omega^1\w\omega^2,\\
  \Xi_1=&Y_{101}\omega^0\w\omega^1+Y_{102}\omega^0\w\omega^2-\ve Y_{002}\omega^1\w\omega^2,\\
  \Xi_2=&(Y_{102}-Y_{012})\omega^0\w\omega^1-Y_{101}\omega^0\w\omega^2+\ve Y_{001}\omega^1\w\omega^2.
\end{aligned}
\end{equation}
The five scalar functions $Y_{001},Y_{002},Y_{012},Y_{101},Y_{102}$ on $\cC$ constitute the entries of the Cotton-York tensor for $[g]$ whose vanishing implies flatness of the conformal structure.  

Lastly, recall that a 2-plane in $T_xM$  is null if its annihilator is a null 1-form with respect to the induced conformal structure on $T^*_xM.$ In the case of Riemannian conformal structures there is no real null 2-plane and in the Lorentzian signature there is an $\RR\PP^1$-parameter family in each tangent space. One can  parameterize (complex or real) null planes in $T_xM$ as
\begin{equation}\label{eq:complex-null-planes-3D}
  \cN^{\KK}_x:=\left\{\Ker\left\{z_0\omega^0+z_1\omega^1+z_2\omega^2\right\}\ \vline\ [z_0\colon z_1\colon z_2]\in\KK\PP^2, z_0^2+\ve z_1^2+\ve z_2^2=0\right\}\!\subset\! \mathrm{Gr}(2,\KK\otimes T_xM).
  \end{equation}
where $\KK$ is $\RR$ or $\CC$  in the case of real or complex null planes, respectively.  Later we will relate $\cN^\CC\backslash\cN^{\RR}$ to the bundle of rays of (future time-like) 1-forms on $M.$

\subsection{CR geometry in dimension 5}\label{sec:cr-structures-type-begin}
 In the definition below a $\delta$-CR structure is a CR or para-CR structure when $\delta$ is -1 or 1, respectively. Moreover, we use the notation $\KK_\delta:=\RR[\sqrt{\delta}]$ i.e.  $\KK_{1}=\RR$ and $\KK_{-1}=\CC.$
 \begin{definition}\label{def:cr-para-cr}
An almost $\delta$-CR structure of hypersurface type on a $(2n+1)$-dimensional manifold $N$ consists of 
a corank one distribution $\scC\subset TN$ equipped with a compatible almost $\delta$-complex structure 
 \begin{equation}\label{JJJ}
\cJ_\delta:\scC\to \scC,\quad  \cJ_\delta^2=\delta\Id, \quad \mathrm{Tr}(\cJ_\delta)=0,\quad
\mathcal{L}(\cJ_\delta X,\cJ_\delta Y)=-\delta\mathcal{L}(X,Y),
 \end{equation}
where $\mathcal{L}:\biw^2\scC\to TN/\scC$ is the Levi bracket of $\scC$ defined at every $x\in N$ as
 \[
\mathcal{L}(X,Y)=[\tilde{X},\tilde{Y}]_x\ \mathrm{mod\ }\scC_x\ \text{ for }\ 
X,Y\in \scC_x,\  \tilde{X},\tilde{Y}\in\Gamma(\scC), \ \tilde{X}_x=X,\ \tilde{Y}_x=Y. 
\]
Denoting  the  $\sqrt{\delta}$ and $-\sqrt{\delta}$ eigenspaces of $\cJ_\delta$  by $\scH,\cscH\subset\KK_\delta\scC:=\scC\otimes\KK_{\delta},$ respectively,  a $\delta$-CR structure is an almost $\delta$-CR structure for which $\scH$ and $\cscH$ are  integrable.

A nondegenerate (almost) $\delta$-CR structure is defined by its Levi bracket being nondegenerate everywhere. If the Levi bracket  has degeneracy, the kernel of $\mathcal{L},$ denoted as $K,$ is equipped with the splitting $\scK\oplus \cscK=K\otimes\KK_\delta,$ where $\scK\subset \scH,\cscK\subset \cscH.$ Define the  higher Levi bracket $\mathcal{L}_2:\scK\otimes \cscH\to \scH/\scK$   by
\begin{equation}\label{eq:2nondeg-Levi-bracket}
  \mathcal{L}_2(X,Y)=[\tilde{X},\tilde{Y}]_x\ \mathrm{mod\ }(\scK_x\oplus \cscH_x)\ \text{ for }\ 
X\in \scK_x, \ Y\in \cscH_x,
 \end{equation}
in which  $\tilde{X}\in\Gamma(\scK)$,  $\tilde{Y}\in\Gamma(\cscH)$,  $\tilde{X}_x=X$,  $\tilde{Y}_x=Y$ with
its conjugate $\overline{\mathcal{L}}_2:\cscK\otimes \scH\to \cscH/\cscK$  defined similarly.
A $\delta$-CR structure is called 2-nondegenerate if $\mathcal{L}_2(X,\cscH)=0$ implies $X=0$ 
and similarly for the conjugate $\overline{\mathcal{L}}_2$.   
 \end{definition}
 \begin{remark}
   Identifying $TN\slash\scC$ at each point $x\in N$ with $\RR,$ the Levi bracket $\cL$ when $\delta=-1$ corresponds to the imaginary part of a Hermitian inner product of signature $(p,q), p+q=n,$ on $\scC\otimes\CC,$ defined up to a conformal factor. The real part of this sub-conformal Hermitian structure on $\scC\otimes\CC$ is given by $g(X,Y)=\cL(X,\cJ_{-1} Y),$ for $X,Y\in\Gamma(\scC),$ which also has signature $(p,q).$ When $\delta=1,$ after composition with the para-complex structure, $\cJ_{1}$, the symmetric bilinear form $g$ on $\scC$ has signature $(n,n)$ and is the real  part of a  subconformal para-Hermitian structure.  The induced action of $\cJ_\delta$ on $\KK_\delta\scC:=\KK_\delta\otimes\scC$ results in the splitting $\KK_\delta\scC=\scH\oplus\cscH,$ where $\scH$ and $\cscH$ are (para-)holomorphic and anti-(para-)holomorphic sub-bundles of $\KK_\delta\scC$ for $\cJ_\delta,$ respectively. Note that as a result of the compatibility condition \eqref{JJJ}, in the para-CR case both $\scH$ and $\cscH$ are null with respect to the induced inner product on $\KK_1\scC\cong \scC.$ 
   Furthermore, the  compatibility condition  in the para-CR case implies that  in the splitting  $\scC=\scH\oplus\cscH$ the integrable distributions $\scH$ and $\cscH$ are Lagrangian  with respect to the induced symplectic structure on $\scC.$ As a result, a para-CR structure is also referred to as an integrable contact Legendrian/Lagrangian/Lagrangean structure. In the definition above when $N$ is 3-dimensional the eigenspaces $\scH$ and $\cscH$ have rank one and, therefore, are automatically integrable. Furthermore, for 2-nondegenerate CR-structures ($\delta=-1$) the nondegeneracy condition on $\bcL_2$   follows automatically, unlike the para-CR case ($\delta=+1$).
Lastly,   when the subconformal $\delta$-Hermitian structure on $\scC$ is nondegenerate, the trace condition in \eqref{JJJ} is automatic for $\delta=- 1$, however when $\delta=1,$ it ensures that the eigenspaces $\scH,\cscH$  have equal rank.
\end{remark}

Solving the equivalence problem for nondegenerate CR structures goes back to \cite{Cartan-CR1,CM-CR,Tanaka-CR,Webster-CR}. Here we recall such a solution in the case of 5-dimensional nondegenerate CR structures of signature (1,1) as  regular and normal parabolic geometries, which is a special case of what appears in \cite[Equation (5.33) and Appendix]{CM-CR} wherein  the term Cartan geometry does not appear. 

\begin{theorem}\label{thm:CR-equiv-problem}     
  A 5-dimensional  CR structure   with Levi form of signature (1,1) is a Cartan geometry  $(\cQ\to N,\phi)$ of type $(\mathrm{SU}(2,2),P)$ where $P$ is the stabilizer of a  null line in $\CC^{2,2}.$   The  Cartan connection and curvature can be expressed as
  \begin{equation}\label{eq:5DCR-Cartan-conn}
         \def\arraystretch{1.2}
\gamma=  \begin{pmatrix}
    -\phi^0_0 & \pi & \sigma &\sigma_0\\
    \beta & \phi^1_1 & \ri\phi^1_2 & \csigma\\
    -\ri\alpha & \ri\phi^2_1 &  -\cphi^1_1 & \cpi\\
    \ri\alpha^0 & \ri\calpha &  \cbeta & \cphi^0_0
  \end{pmatrix} ,\quad    %
  \exd\gamma+\gamma\w\gamma=
   \begin{pmatrix}
     -\Phi^0_0 & \Pi & \Sigma & \ri\Sigma_0 \\
   0 & \Phi^1_1 & \ri\Phi^1_2 & \overline\Sigma\\
     0 & \ri\Phi^2_1 &  -\overline\Phi^1_1 & \overline\Pi\\
     0 & 0 &  0 & \overline\Phi^0_0
   \end{pmatrix}
   \end{equation}
 which  are $\mathfrak{su}(2,2)$-valued with respect to the Hermitian  form  
 \[H(Z)=Z_4\overline Z_1-Z_1\overline Z_3-Z_3\overline Z_1+Z_1\overline Z_4,\quad Z=(Z_1,Z_2,Z_3,Z_4)\in\CC^4,\]
 where  $\alpha=\alpha^1+\ri\alpha^2$ and $\beta=\beta^1+\ri\beta^2,$ 
 $\Im\Phi^1_1=0,$ and
 \[
   \begin{aligned}
     \Phi^2_1\equiv& C_0\beta^1\w\beta^2+C_1(\alpha^1\w\beta^2-\alpha^2\w\beta^1)+C_2\alpha^1\w\alpha^2, \\
     \Re\Phi^1_1\equiv& C_1\beta^1\w\beta^2+C_2(\alpha^1\w\beta^2-\alpha^2\w\beta^1)+C_3\alpha^1\w\alpha^2, \\
     \Phi^1_2\equiv & C_2\beta^1\w\beta^2+C_3(\alpha^1\w\beta^2-\alpha^2\w\beta^1)+C_4\alpha^1\w\alpha^2, \\
   \end{aligned}
 \]
modulo $\{\alpha^0\},$ for  real-valued functions $C_0,\cdots,C_4$  on $\cQ.$  The fundamental (harmonic) curvature, also referred to as the Chern-Moser curvature, can be represented by a    weighted binary holomorphic quartic $C\in\Gamma(\mathrm{Sym}^4\scC^*)\otimes(\biw^2\scH)^{-1}\otimes(\scC^{\perp})^{\half}, $ where, for any section $s\colon N\to\cQ,$
 \begin{equation}\label{eq:W-CR-quartic}
   C=s^*\!\!\left(C_4(\alpha)^4 +4C_3(\alpha)^3\beta +6C_2(\alpha)^2(\beta)^2 +4C_1\alpha(\beta)^3 +C_0(\beta)^4\right)\!\otimes\! \left(\tfrac{\partial}{\partial s^*\alpha}\w \tfrac{\partial}{\partial s^*\beta}\right)^{-1}\!\!\!\!\otimes \left(s^*\alpha^0\right)^{\half}.
 \end{equation}
 The vanishing of $C$ implies  $\Phi=0,$ i.e. $N$ is locally equivalent to the indefinite hyperquadric  in $\CC^3.$ 
\end{theorem}
The full Cartan curvature will not be needed here, e.g. see \cite[Appendix]{CM-CR} for more details. 
\begin{remark}\label{rmk:cr-geom-remarks}
  Note that the weighted holomorphic quartic \eqref{eq:W-CR-quartic} is invariantly defined on the contact distribution at each point. Similarly, its conjugate defines a weighted anti-holomorphic quartic on the contact distribution, namely, $\overline C\in \Gamma(\mathrm{Sym}^4\scC^*)\otimes(\biw^2\cscH)\otimes(\scC^{\perp})^{-\half},$ where
\[\overline C=s^*\left(C_4(\calpha)^4 +4C_3(\calpha)^3\cbeta +6C_2(\calpha)^2(\cbeta)^2 +4C_1\calpha(\cbeta)^3 +C_0(\cbeta)^4\right)\otimes \left(\tfrac{\partial}{\partial s^*\calpha}\w \tfrac{\partial}{\partial s^*\cbeta}\right)\otimes \left(s^*\alpha^0\right)^{-\tfrac 12}.\]
\end{remark}

The fundamental binary quartic  of a indefinite CR 5-manifold   can be interpreted as an obstruction to the integrability of an $\SSS^1$-bundle of   Lagrangian 2-planes in $\scC$ that are null with respect to the indefinite subconformal structure on $\scC,$ as was first observed in \cite{Bryant-HolomorphicCurves}. More generally, almost CR structures of signature $(1,1)$ for which the binary quartic above is zero has been studied in-detail in \cite{KM-subconformal}, where it is shown that they are equipped with a 4-parameter family of foliations by null and Lagrangian 2-planes,  providing a 5-dimensional analogue of half-flatness in 4-dimensional conformal geometry,

Recall that in a non-degenerate CR structure of hypersurface type if $\alpha^0$ is a contact 1-form, i.e. $\scC=\Ker\{\alpha^0\},$ then $\rho:=\exd\alpha^0$ defines a symplectic 2-form on $\scC,$ which is related to the Levi-bracket via $\alpha^0(\cL(X,Y))=\alpha^0([X,Y])=-\rho(X, Y).$ The corresponding sub-Hermitian metric, also referred to as the Levi form, is defined as $h(X,Y)=\frac{1}{2\ri}\rho(X,\bar Y)$ for $X,Y\in\Gamma(\scC\otimes\CC).$  Using the almost complex structure $\cJ\in\mathrm{Aut}(\scC),$ with $\cJ^2=-\mathrm{Id},$ the  symmetric bilinear form $g(X,Y):=\rho(X,\cJ Y),$ for $X,Y\in\Gamma(\scC),$ is the real part of $h$ and  has signature $(p,q).$ In general, since there is no distinguished contact 1-form, the symplectic 2-form, $\rho,$ and the symmetric bilinear form, $g,$ are defined up to a conformal factor.  In  CR structures  with Levi form of signature (1,1), in terms of the Cartan connection \eqref{eq:5DCR-Cartan-conn},  the conformal symplectic structure $[\rho]\subset\biw^2\scC^*$ and indefinite subconformal structure $[g]\subset\mathrm{Sym}^2(\scC^*)$ are given by
\begin{equation}\label{eq:sympl-subconf}
  \rho=\alpha^1\w\beta^1+\alpha^2\w\beta^2=\half\alpha\w\cbeta+\half\calpha\w\beta,\quad g=\alpha^1\beta^2-\alpha^2\beta^1=\tfrac{\ri}{2}(\alpha\cbeta-\calpha\beta),
\end{equation}
where, by abuse of notation, we have suppressed the pull-back of the 1-forms in \eqref{eq:5DCR-Cartan-conn} by a section $s\colon N\to\cQ,$ and, as stated in Theorem \ref{thm:CR-equiv-problem}, $\alpha=\alpha^1+\ri\alpha^2,$ $\beta=\beta^1+\ri\beta^2$ belong to the dual $\ri$-eigenspace of $\cJ,$ and are referred to as the holomorphic 1-form on $\scC\otimes\CC.$

As a result, the set of 2-planes in $\scC$ that are null with respect to $g$ define two   $\PP^1$-bundles $\cA,\cB\subset\mathrm{Gr}(2,TN)$ over $N$ whose fibers at a point $x\in N$ can be parametrized as follows
\begin{equation}\label{eq:alpha-beta-planes}
  \begin{aligned}
    \cA_x:=&\left\{\Ker\left\{\alpha^0,a_2\alpha^1-a_1\beta^1,a_2\alpha^2-a_1\beta^2\right\}\ |\ [a_1:a_2]\in \RR\PP^1\right\}\subset \mathrm{Gr}(2,T_xN),\\      
    \cB_x:=&\left\{\Ker\left\{\alpha^0,b_1\alpha^1-b_2\alpha^2,b_1\beta^1-b_2\beta^2\right\}\ |\ [b_1:b_2]\in \RR\PP^1\right\}\subset \mathrm{Gr}(2,T_xN).
      \end{aligned}
    \end{equation}
    By analogy with 4-dimensional indefinite conformal structures, we refer to the family $\cA_x$ and $\cB_x$ as $\alpha$-planes and $\beta$-planes at $x.$ Furthermore, we note that all $\alpha$-planes are Lagrangian with respect to $[\rho].$ However, among $\beta$-planes there is no real Lagrangian plane. The complexified $\beta$-planes for which $[b_1:b_2]$ is $[1:\ri]$ and $[1:-\ri]$ correspond to $\scH,\cscH\subset \scC,$ respectively, and are the only Lagrangian planes among the complexified $\beta$-planes, whose integrability encodes  the complex structure on $\scC$.

    Using the dual frame, every $\alpha$-plane $P_a\in\cA_x,$ $a=[a_1:a_2]$ is given by $\langle a_1\partial_{\alpha^1}+a_2\partial_{\beta^1},a_1\partial_{\alpha^2}+a_2\partial_{\beta^2}\rangle,$ and, thus, contains a unique contact holomorphic  direction $\langle V_a\rangle\subset P_a^\CC$ where
    \[V_a= a_1\partial_{\alpha}+a_2\partial_{\beta}\in\Gamma(\scH_x),\]
and $\partial_{\alpha}=\half (\partial_{\alpha^1}-\ri\partial_{\alpha^2})$ and  $\partial_{\beta}=\half(\partial_{\beta^1}-\ri\partial_{\beta^2}).$    Evaluating the holomorphic quartic $C$ on  holomorphic vector fields in complexified $\alpha$-planes at $x\in N$ results in a weighted binary quartic. More explicitly, if  $a=[a_1:a_2]\in\PP^1,$ and $V_a\in P_a^\CC$ is a contact holomorphic vector field, then one can define
    \begin{equation}\label{eq:quartic}
      Q_a=C(V_a,V_a,V_a,V_a)=\lambda(C_4a_1^4+4C_3a_1^3a_2+6C_2a_2^2a_1^2+4C_1a_1a_2^3+C_0a_2^4),
          \end{equation}
for some non-zero scalar  $\lambda$. By our discussion above, the roots of the binary quartic $Q_a$ determine at most four Lagrangian null planes at $x$, which are used in  the following definition.
    \begin{definition}\label{def:principal-null-plane}
A Lagrangian null  plane  is called a principal null plane of multiplicity $k$ if it corresponds to a real root of the binary quartic $Q_a$ in \eqref{eq:quartic} with multiplicity $k$.
\end{definition}
It is straightforward to verify that the notion of a principal null plane is independent of the choice of parametrization of $\alpha$-planes. The terminology is in analogy with the notion of a principal null direction in 4-dimensional Lorentzian geometry.
\begin{remark}\label{rmk:principal-null-planes}
 Naturally, if the binary quartic $C$ is identically zero, i.e.  the CR structure is flat, then every $\alpha$-plane is a principal null plane. If the binary quartic is non-zero, then  at each point there are at most four principal null planes.  These properties suggest a notion of Petrov type for 5-dimensional CR structures, based on the number of roots of the binary quartic and their multiplicity and reality. The only important case in this article is the case when the quartic has a real root of multiplicity four, which would be the analogue of Petrov type $\sN$ is the context of 4-dimensional Lorentzian conformal geometry. 
\end{remark}

    Since any  integrable  distribution of contact 2-planes on $N$ have to be Lagrangian, assuming that they are also null with respect to $[g]$ implies that the contact 2-planes are $\alpha$-planes. The next proposition shows that being principal is a necessary condition for an integrable field of  null contact 2-planes, which has been first observed in \cite{Bryant-HolomorphicCurves}. 
\begin{proposition}\label{prop:quartic-integrability}
Given an indefinite  CR 5-manifold $(N,J,\scC),$ if $\scD\subset\scC$ is an integrable distribution of rank 2 that is null with respect to the induced subconformal structure on $\scC,$ then $\scD$ is  a principal null plane everywhere.
\end{proposition}
\begin{proof}
  Since  $\scD\subset \scC$ is  null and contact, as was mentioned above, its integrability implies that it is   everywhere Lagrangian. Thus, it defines a section $s\colon N\to\cA$ of $\alpha$-planes given by $\Ker\{\alpha^0,a_2\alpha^1-a_1\beta^1,a_2\alpha^2-a_1\beta^2\},$ for functions $a_2,a_1\in C^\infty(N,\RR).$ %
  Let us explore differential consequences for the integrability of Pfaffian system $\cI_\lambda:=\{\alpha^0,\alpha^1-\lambda\beta^1,\alpha^2-\lambda\beta^2\}$, where $\lambda=a_1/a_2\in C^{\infty}(N,\RR\PP^1)$.

  Expressing the   integrability of $\cI_\lambda$ using the Cartan curvature \eqref{eq:5DCR-Cartan-conn}, it follows that $\exd\alpha^0\equiv 0$ modulo $\cI_\lambda$. Furthermore,  working modulo the ideal $\cI_\lambda$ allows one to set $\alpha^0\equiv 0$ and $\alpha^i\equiv \lambda\beta^i$ for $i=1,2.$ Hence, one obtains 
  \[
    \begin{aligned}
      \exd(\alpha^1-\lambda\beta^1)\equiv& \beta^1\w (\exd\lambda-\lambda^2\phi^1_2-2\lambda\Re\phi^1_1-\phi^2_1), \\
      \exd(\alpha^2-\lambda\beta^2)\equiv&  \beta^2\w (\exd\lambda-\lambda^2\phi^1_2-2\lambda\Re\phi^1_1-\phi^2_1 ). 
    \end{aligned}\]
  modulo $\cI_\lambda.$ It follows that,  modulo $\cI_\lambda,$
  \begin{equation}\label{eq:dlmabda}
    \exd\lambda\equiv \lambda^2\phi^1_2+2\lambda\Re\phi^1_1+\phi^2_1.
  \end{equation}
Examining the identity $\exd^2\lambda=0$ together with \eqref{eq:dlmabda} and   the Cartan curvature \eqref{eq:5DCR-Cartan-conn}, one obtains
  \[(C_4\lambda^4+4C_3\lambda^3+6C_2\lambda^2+4C_1\lambda+C_0)\beta^1\w\beta^2\equiv 0.\]
  The algebraic condition above   implies that $\lambda$ is a root of the binary quartic \eqref{eq:quartic}, expressed in affine coordinates $a_1/a_2,$ i.e. the integrability of $\cI_\lambda$ implies that $\lambda$ is   a principal null plane everywhere.   
\end{proof}

\subsection{Twistor CR manifolds and their Petrov type}  
\label{sec:lebr-feff-constr}
The twistor bundle of an even-dimensional geometric structure   is generally thought of as the bundle of  compatible almost complex structures, e.g. see \cite{Twistors1,Twistors2}.  The  twistor bundle of  3-dimensional oriented conformal structures, denoted as $\cT,$  is defined in \cite{LeBrun-CR,MasonPhD} as the bundle of projectivized complex null planes. Alternatively,  noting that every (time-like) plane carries a canonical almost complex structure via the pull-back of the (Lorentzian) Riemannian conformal structure, one expects the bundle of positively oriented (time-like) planes, which can be identified with the bundle of rays of (future time-like)  1-forms, to serve the role of $\cT$. We briefly overview these two constructions, starting with the latter.

Define $\cT=\RT,$  where $\RT$ denotes the bundle of rays in $T^*M$ with respect to a Riemannian conformal structure, i.e. the  two-fold cover of $\PP T^*M$, while in the Lorentzian signature it denotes the rays of future time-like 1-forms. Thus, the twistor bundle $\widetilde\mu\colon\cT\to M$ is a 5-manifold whose fibers are the 2-sphere, $\SSS^2,$ if $[g]$ is Riemannian or the 2-disk, $\DD^2,$ if $[g]$ is Lorentzian. To view $\hat\mu\colon\cC\to\cT$ as an associated bundle of $\mu\colon\cC\to M,$  recall that  $(\cC\to M,\psi)$ has type $(\mathrm{SO}^o(p+1,q+1),Q_1),$ where $(p,q)$ is either (3,0) or (2,1), $Q_{1}=Q_0\ltimes\RR^{3},$ $Q_0=\mathrm{CO}^o(p,q)\cong R^+\times\mathrm{SO}^o(p,q)$ is the structure group, $\mathrm{SO}^o(2,1)$ is the identity component of $\mathrm{SO}(2,1)$ and  $\mathrm{SO}^o(3,\RR)=\mathrm{SO}(3,\RR).$   The stabilizer of  a positively oriented  (time-like) plane defines a subgroup $S\subset Q_1.$ Thus, one obtains that  $\hat\mu\colon\cC\to\cT$ is a principal $S$-bundle i.e.
  \begin{equation}\label{eq:Tg-bundle}
    \cT\cong\cC\slash S:=\cC\times_{Q_1}(Q_1\slash S),\quad \text{where}\quad  S=\mathrm{CO}(2,\RR)\ltimes \RR^3.
  \end{equation}
Consequently, one recovers the fact that the fiber of $\widetilde\mu\colon\cT\to M$  at a point $x\in M$ is
\begin{equation}\label{eq:fibers-cT}
     \cT_x=Q_1\slash S=\mathrm{SO}^o(p,q)\slash \mathrm{SO}(2,\RR)\cong \left\{
      \begin{aligned}
        &\SSS^2\quad (p,q)=(3,0),\\
        &\DD^2\quad (p,q)=(2,1).\\
      \end{aligned}\right.
  \end{equation}

  \begin{remark}\label{rmk:twistor-paracr-manifold}
In Lorentzian signature,  the action of $Q_1$ on each oriented projectivized cotangent space has five orbits given by future and past rays of time-like 1-forms, rays of space-like 1-forms, and  rays of future and past null 1-forms. 
 One  can  define a  twistor bundle using space-like planes in each tangent space equipped with their canonically induced almost para-complex structure. The construction goes through as before  with the only difference being the fibers which are  $\mathrm{SO}(2,1)\slash\mathrm{SO}(1,1)\cong \RR\times\RR\PP^1.$ 
\end{remark}

As was mentioned above,  $\cT$ was originally defined as  a connected component of the set of complexified null 2-planes without any real null 2-planes, defined in \eqref{eq:complex-null-planes-3D}, i.e. $\cN^{\CC}$ in the Riemannian signature and a connected component of  $\cN^\CC\backslash\cN^\RR$ in the Lorentzian signature. The two definitions are equivalent since there is a bijection between  (time-like) planes equipped with their canonically induced almost complex structure and projective classes of complex null 1-forms  $[\omega_z]:=[z_0\omega^0+z_1\omega^1+z_2\omega^2]$ in  \eqref{eq:complex-null-planes-3D}. To see this, note that given a positive (time-like) plane $p\in\mathrm{Gr}^+(2,T_xM)$, one can find a coframe $(\omega^0,\omega^1,\omega^2)$ that satisfies \eqref{eq:conformal-metric} and $p=\Ker\{\omega^0\}.$ The induced $\mathrm{SO}(2,\RR)$ action on $(\omega^1,\omega^2)^t$ defines an almost complex structure whose dual $\ri$-eigenspace is $\omega^1+\ri\omega^2,$ which is a null complex 1-form. Conversely, given a projective class of a complex null 1-form, $[\omega_z]:=[z_0\omega^0+z_1\omega^1+z_2\omega^2],$ $z_0^2+\ve z_1^2+\ve z_2^2=0,$ one can find $[A^i_j]\in \mathrm{SO}^o(p,q)$ such that $[\omega_z]=[\tilde\omega^1+\ri \tilde\omega^2]$ where $\tilde\omega^i=A^i_j\omega^j.$ More explicitly, examining the case $(p,q)=(3,0)$ over the affine chart  $z_2=\ri,$ one has $[\omega_z]=[z_0\omega^0+z_1\omega^1+\ri\omega^2]$ and $z_0^2+z_1^2=1.$ Using the 2-to-1 isogeny  $\mathrm{SU}(2)\to \mathrm{SO}(3,\RR),$ the coframe transformation $\omega^i\to\tilde\omega^i$  is 
\[
  \begin{pmatrix}   
         \ri\tilde\omega^1 &  \ri\tilde\omega^2+\tilde\omega^0\\
         \ri\tilde\omega^2- \tilde\omega^0 & -\ri\tilde\omega^1 \\      
       \end{pmatrix}  =
       \begin{pmatrix}   
          a & -\overline b  \\
          b& \overline a \\      
    \end{pmatrix}  
       \begin{pmatrix}   
         \ri\omega^1 &  \ri\omega^2+\omega^0\\
         \ri\omega^2- \omega^0 & -\ri\omega^1 \\      
       \end{pmatrix}
              \begin{pmatrix}   
          \overline a & \overline b  \\
          -b& a \\      
    \end{pmatrix},  
\]
where $|a|^2+|b|^2=1.$ Finding $a,b\in\CC$ such that $[\tilde\omega^1+\ri\tilde\omega^2]=[z_0\omega^0+z_1\omega^1+\ri\omega^2]$ results in a set of  polynomial relations  that can be easily solved for any $z_0,z_1\in\CC$ satisfying $z_0^2+z_1^2=1.$ This shows that $\omega_z$ gives the dual $\ri$-eigenspace for the almost complex structure on $\Ker\{\tilde\omega^0\}.$

Now we state the main result of this section. %
\begin{proposition}\label{prop:algebr-type-CR-twist}
  The twistor bundle, $\cT,$ of a conformal 3-manifold $(M,[g])$ is an indefinite nondegenerate CR 5-manifold whose fundamental binary quartic  is zero when the conformal structure is flat or has a repeated root of multiplicity four in which case the corresponding principal null plane of multiplicity four  is the vertical distribution in the fibration $\cT\to M$. 
\end{proposition}
\begin{proof}
To show that $\cT$ carries a CR structure,  we  follow  \cite{SY-LieContact1}, which is different from \cite{LeBrun-CR} and is based on exploiting the embedding of Lie algebras $\mathfrak{so}(p+1,q+1)\hookrightarrow \mathfrak{su}(2,2),$ where $p+q=3,$ given explicitly in \eqref{eq:sopq-su22}.  Firstly, using the description   $\cT=\cC\slash S$ in \eqref{eq:Tg-bundle} together with the Cartan connection \eqref{eq:conf-cartan-conn},  one obtains that the 1-forms $(\omega^0,\omega^1,\omega^2,\theta^1,\theta^2)$ are semi-basic with respect to the fibration $\cC\to \cT.$  Furthermore, by the structure  equations  \eqref{eq:Cartan-curvature-conformal-conn}, one knows
  \[\exd\omega^0\equiv \omega^1\w\theta^1+\omega^2\w\theta^2\mod \{\omega^0\}.\]
It is a straightforward task to check that the complex Pfaffian system $\{\omega^0,\omega^1+\ri\omega^2,\theta^1+\ri\theta^2\}$ is  integrable. As a result, the contact distribution $\scC=\Ker\{\omega^0\}\subset T\cT$ is equipped with an integrable almost complex structure. This almost complex structure satisfies the compatibility condition with the Levi bracket, expressed in \eqref{JJJ}, because of the relation
  \[\exd\omega^0\equiv \half(\omega\w\ctheta+\comega\w\theta),\mod\{\omega^0\},\]
  where $ \omega=\omega^1+\ri\omega^2$ and $\theta=\theta^1+\ri\theta^2.$ The conformal class of  Hermitian inner products  $[h]\subset \mathrm{Sym}^2(\scC^*)$ is given by $h=\tfrac{\ri}{2}(\omega\ctheta-\comega\theta).$ Using an adapted coframe $(\omega^0,\omega,\theta,\comega,\ctheta),$ one can find the Cartan connection for the induced nondegenerate CR structure on $\cT.$ Namely, the Lie algebra embedding $\mathfrak{so}(p+1,q+1)\hookrightarrow \mathfrak{su}(2,2)$  results in the injection $\iota\colon S\to P,$ where $P\in \mathrm{SU}(2,2)$ is the parabolic subgroup in Theorem \ref{thm:CR-equiv-problem}  and $S$ is given in \eqref{eq:Tg-bundle}. Defining $\cQ=\cC\times_{S}P,$ one obtains a natural inclusion $\iota\colon\cC\to\cQ$ between the Cartan geometries $(\cC\to M,\psi)$ and $(\cQ\to\cT,\gamma)$ via which
  \begin{equation}\label{eq:sopq-su22}
      \def\arraystretch{1.3}
\iota^*\gamma=  \begin{pmatrix} 
         -\half\cphi &  -\tfrac 14\ve\ctheta & \half\ri\cxi  &    \tfrac 14\ri\xi^0\\
    \theta  & -\half\phi & \ve\ri\xi^0  & -\half\ri\xi \\
    -\ri\omega &  -\half\ve\ri\omega^0 &  \half\cphi  & -\tfrac 14\ve\theta  \\
    -2\ri\omega^0&  \ri\comega &  \ctheta  &  \half\phi
  \end{pmatrix},
    \end{equation}
where the 1-forms
\begin{equation}\label{eq:holom-1-forms-CR-cT}
  \omega=\omega^1+\ri\omega^2,\quad \xi=\xi^1+\ri\xi^2,\quad \theta=\theta^1+\ri\theta^2,\quad \phi=\theta^0+\ri\theta^3,
  \end{equation}
are defined in terms of the entries of the Cartan connection \eqref{eq:conf-cartan-conn} for the 3-dimensional conformal structure $[g]$.
 Using $\iota^*\gamma$ and the structure equations \eqref{eq:Cartan-curvature-conformal-conn} and \eqref{eq:5DCR-Cartan-conn},  the coefficients $C_0,\cdots,C_4$  of the holomorphic binary quartic \eqref{eq:W-CR-quartic} on $\cT,$  pulled-back via $\iota,$  are found to be
 \begin{equation}\label{eq:Cis-quartic}
   C_0=C_1=C_2=C_3=0,\quad C_4=\ve Y_{012}.
    \end{equation}
The Cotton-York entry $Y_{012}$ is identically zero everywhere on $\cC$ if and only if the conformal structure is flat. As a result, the CR structure on $\cT$ is flat if and only if the conformal structure $[g]$ is flat. Otherwise, from \eqref{eq:Cis-quartic}, it follows that the binary quartic \eqref{eq:W-CR-quartic} has a repeated root of multiplicity four. Using the parametrization of $\alpha$-planes \eqref{eq:alpha-beta-planes} and the quartic \eqref{eq:quartic}, it follows that the repeated root corresponds to the $\alpha$-plane for which $[a_1:a_2]=[0:1],$ i.e. the contact planes $\Ker\{\omega^0,\omega^1,\omega^2\},$ whose integral  surfaces are the fibers of $\cT\to M.$ 
\end{proof}

As was mentioned before Proposition \ref{prop:algebr-type-CR-twist}, in \cite{LeBrun-CR, MasonPhD} the induced CR structure on $\cT$ is obtained by identifying it with $\cN^{\CC}\backslash\cN^\RR.$   The 1-form $\omega^1+\ri\omega^2$ on $\cT$ is  the horizontal lift of the dual  $\ri$-eigenspace for the almost complex structure on  positively oriented (future time-like) planes using any Weyl connection $\nabla^{[g]}$. The 1-form $\theta^1+\ri\theta^2$ is the canonical holomorphic 1-form on $\mathrm{SO}^o(p,q)\slash\mathrm{SO}(2,\RR)$ that is invariant under the action of $\mathrm{SO}(2,\RR).$ It is straightforward to show that the resulting  complex structure on the contact distribution is independent of the choice of  $\nabla^{[g]}.$

The twistor CR structures were also discovered independently in \cite{SY-LieContact1} as the 3-dimensional case of geometric structures referred to as \emph{Lie contact structures} which give rise to an analogue of the Fefferman construction for conformal structures in all dimensions.  We refer the reader to \cite[Section 4.5.3]{CS-Parabolic} for an overview of this construction.

\begin{remark}\label{rmk:petrov-type-twistor}
As was mentioned before, in terms of the adapted coframe  \eqref{eq:holom-1-forms-CR-cT} on $\cT,$ the value  $\lambda:=a_1/a_2=0$ in the parametrization  \eqref{eq:alpha-beta-planes},  corresponds to the fibers of $\cT\to M.$   Thus, by Proposition \ref{prop:quartic-integrability} one knows a priori that  the pull-back of the binary quartic of $\cT$ to $\cC$  vanishes at $\lambda=0.$ The fact that  $\cT$  has a repeated root of multiplicity at least four is an analogue of the fact that the Weyl curvature of 4-dimensional Lorentzian Fefferman conformal structures, defined for any CR 3-manifold, has \textrm{Petrov type $\sN$} \cite{Lewandowski}.  Similarly, it turns out the twistor bundle defined in Remark \ref{rmk:twistor-paracr-manifold}  carries a nondegenerate \emph{para-CR structure} whose fundamental binary quartic satisfies the same properties as in Proposition \ref{prop:algebr-type-CR-twist}. 
\end{remark}
\section{Projective structures and 2-nondegenerate CR geometry}
\label{sec:from-3d-projective}
This section gives a quick review of (oriented) projective structures, their Cartan geometric description, and the corresponding  almost para-CR structure on their (oriented) projectivized cotangent bundle. In \ref{sec:twistor-bundle-2} we describe a twistor bundle for 3-dimensional projective structures and its canonical 2-nondegenerate CR structure of hypersurface type. 
 
\subsection{Projective structures in dimension three}\label{sec:proj-struct-dimens}
A projective structure $[\nabla]$ on a  manifold $M$ is given by an equivalence class of torsion-free connections on $M$ whose geodesics coincide as unparametrized curves. If $M$ is equipped with an orientation, then $[\nabla]$ is called an oriented projective structure.   Let $\gamma(t)=(x^i(t))\subset M$ be a geodesic for $[\nabla].$ Denoting  its tangent vector field  as $\dot\gamma(t)=( x_t^i(t)),$ if $\nabla,\tilde\nabla\in [\nabla],$   then there is  a function $v(t)$ along $\gamma(t)$ such that
\begin{equation}\label{eq:proj-equiv}
  (\tilde \nabla_{\dot\gamma(t)}-\nabla_{\dot\gamma(t)})\dot\gamma(t)=v(t)\dot\gamma(t)\Rightarrow (\tilde\Gamma^i_{jk}-\Gamma^i_{jk})x_t^j(t) x_t^k(t)=v(t)x_t^i(t)\Rightarrow \tilde\Gamma^i_{jk}=\Gamma^i_{jk}+\delta^i_{j}f_k+\delta^i_{k}f_j,
  \end{equation}
for  some functions $(f_1,\cdots,f_n)$ on $M$ where $\Gamma^{i}_{jk}$ and $\tilde\Gamma^i_{jk}$ are the Christoffel symbols of $\nabla$ and $\tilde\nabla,$ respectively.  
    
Solving the equivalence problem of projective structures goes back to Cartan \cite{Cartan-proj}. The solution is given as Cartan geometries of type $(\mathrm{SL}(n+1,\RR),P_1),$ where $P_1=\mathrm{GL}(n,\RR)\ltimes\RR^n\subset \mathrm{SL}(n+1,\RR)$ is the parabolic subgroup that stabilizes a line in $\RR^{n+1}.$ In the case of oriented projective structures $P_1=\mathrm{GL}^+(n,\RR)\ltimes\RR^n$ is the stabilizer of a ray. We are interested in  3-dimensional projective structures, which correspond to regular and normal Cartan geometries  $(\pi\colon\cP\to M,\varphi)$ of type $(\mathrm{SL}(4,\RR),P_1).$ The Cartan connection and Cartan curvature, which can be found in \cite{KN-projective},  are given by
\begin{equation}\label{eq:3Dproj-CartanCurv}
  \def\arraystretch{1.2}
\varphi=  \begin{pmatrix}
    -\eta^i_i & \mu_2 & \mu_1 &\mu_0\\
    \tau^2 & \eta^2_2 & \eta^2_1 & \eta^2_0\\
    \tau^1 & \eta^1_2 &  \eta^1_1 & \eta^1_0\\
    \tau^0 & \eta^0_2 &  \eta^0_1 & \eta^0_0
  \end{pmatrix} ,\quad    %
  \bvarphi=\exd\varphi+\varphi\w\varphi=
   \begin{pmatrix}
     0 & \bmu_2 & \bmu_1 & \bmu_0 \\
     0 & \beeta^2_2 & \beeta^2_1 & \beeta^2_0\\
     0 & \beeta^1_2 &  \beeta^1_1 & \beeta^1_0\\
     0 & \beeta^0_2 &  \beeta^0_1 & \beeta^0_0
   \end{pmatrix},
\end{equation}
where
\begin{equation}\label{eq:proj-weyl}
  \beeta^i_j=\half W^i_{jkl}\tau^k\w\tau^l,\quad \bmu_j=\half W_{jkl}\tau^k\w\tau^l,\quad W^i_{jik}=0,\quad W^i_{[jkl]}=0,\quad W_{[ijk]}=0.
  \end{equation}
The structure functions $W^i_{jkl}$ on $\cP$ are the entries of the  \emph{projective Weyl curvature} whose vanishing imply flatness, i.e. the projective structure is locally equivalent to the one defined by projective  lines in $\PP^3.$

The 1-forms $\tau^0,\tau^1,\tau^2$ are semi-basic with respect to the fibration $\cP\to M.$  Furthermore, the leaves of the integrable Pfaffian system $\cH=\{\tau^0,\tau^1,\eta^0_2,\eta^1_2\}$ on $\cP$ project to the geodesics of $[\nabla]$ on $M,$ which in the model space $\PP^3$ can be viewed as follows: if $\ell\subset\PP^3$  is the the projectivization of the 2-plane $\langle e_3,e_2\rangle\subset\RR^4,$ where $(e_3,e_2,e_1,e_0)$ is  the standard basis for $\RR^4,$ the Lie algebra of the stabilizer of $\ell$ is given by $\mathfrak{sl}(4,\RR)$-valued matrices for which the bottom left $2\times 2$ block is zero, i.e. the stabilizer is  the parabolic subalgebra $\mathfrak{p}_2\subset\mathfrak{sl}(4,\RR)$ defined by crossing the second node in the Dynkin diagram of $A_3.$ In terms of the Cartan connection \eqref{eq:3Dproj-CartanCurv}, the stabilizer of $\ell$ corresponds to the Pfaffian system $\cH.$
In the curved case, this can be seen as follows: the 5-dimensional leaf space of the Pfaffian system $\cK=\{\tau^0,\tau^1,\tau^2,\eta^0_2,\eta^1_2\}$ can be identified with $\tilde\pi\colon\PP TM\to M,$ since the rank 2 vertical distribution  $\scU=\Ker\{\tau^0,\tau^1,\tau^2\}$ together with the tautological rank 3 distribution $\scT=\Ker\{\tau^0,\tau^1\}$ satisfy $[\scU,\scT]= T\PP TM.$  Recall that the rank 3 tautological distribution  at $p=(x,[v])\in\PP TM,$ where $x\in M$ and $[v]\subset T_xM$ is a line, is defined as $\scT=\tilde\pi_*^{-1}([v])\subset T_p\PP T M.$  Thus, the distributions $(\scU,\scT)$ define the canonical \emph{multi-contact structure} on $\PP TM,$ i.e. the leaf space of $\cK$ is locally equivalent to $\PP TM.$
If $\tilde\scE\subset TTM$ denotes the rank 3 horizontal distribution defined by $\nabla\subset[\nabla],$ then $\ell:=\scT\cap\tilde\pi_*(\tilde\scE)\subset T\PP TM$ is the line field $\ell=\Ker\cH=\Ker\{\tau^0,\tau^1,\eta^0_2,\eta^1_2\},$ referred to as the \emph{geodesic spray} of $\nabla,$  which is invariant of the choice of connection $\nabla,$ and, via projection to $M,$ its integral curves exhaust the geodesics of $[\nabla]$ in a one-to-one fashion.

The geodesic equation   for any representative connection  $\nabla\in[\nabla]$ gives a triple of second order ODEs $x^i_{tt}=-\Gamma^i_{jk}x^j_tx^k_t$,  where the independent parameter $t$ is defined up to homotheties and translations. Considering geodesics as unparametrized curves, a local coordinate on $M$ can be taken as the independent parameter, which allows one to arrive at a pair of second order ODEs, defined up to \emph{point transformations}, as follows. 
Working locally with an open set of geodesics along which $ x^0_t\neq 0,$ one can set $t=t(x_0)$ and $x^i=x^i(x^0)$ for $i=1,2,$  along such geodesics.   Defining    $\dot x^i=\tfrac{\exd x^i}{\exd x^0}=\tfrac{x^i_t}{x^0_t}$ and $\ddot x^i=\tfrac{\exd^2 x^i}{(\exd x^0)^2}=\tfrac{1}{x^0_t}\tfrac{\exd}{\exd t}(\tfrac{x^i_t}{x^0_t})=\tfrac{x^i_{tt}}{(x^0_t)^2}-\tfrac{x^i_tx^0_{tt}}{(x^0_t)^3},$   the geodesic equation reduces to the pair of 2nd order ODEs
\begin{equation}\label{eq:pair-2nd-ODEs}
   \begin{aligned}
\ddot{x}^1=
&\,\, \Gamma^0_{11}(\dot{x}^1)^3+2\Gamma^0_{12}(\dot{x}^1)^2\dot{x}^2+\Gamma^0_{22}\dot{x}^1(\dot{x}^2)^2
	+(2\Gamma^0_{01}-\Gamma^1_{11})(\dot{x}^1)^2\\
& +2(\Gamma^0_{02}-\Gamma^1_{12})\dot{x}^1\dot{x}^2-\Gamma^1_{22}(\dot{x}^2)^2
	+(\Gamma^0_{00}-2\Gamma^1_{01})\dot{x}^1-2\Gamma^1_{02}\dot{x}^2-\Gamma^1_{00},\\
\ddot{x}^2=
&\,\, \Gamma^0_{11}(\dot{x}^1)^2\dot{x}^2+2\Gamma^0_{12}\dot{x}^1(\dot{x}^2)^2+\Gamma^0_{22}(\dot{x}^2)^3
	-\Gamma^2_{11}(\dot{x}^1)^2+2(\Gamma^0_{01}-\Gamma^2_{12})\dot{x}^1\dot{x}^2\\
& +(2\Gamma^0_{02}-\Gamma^2_{22})(\dot{x}^2)^2-2\Gamma^2_{01}\dot{x}^1
	+(\Gamma^0_{00}-2\Gamma^2_{02})\dot{x}^2-\Gamma^2_{00}.
\end{aligned}
\end{equation}
The pair of second order ODEs above are defined up to local diffeomorphisms or, equivalently, up to the change of variables $x^i\mapsto \tilde x^i= X^i(x^0,x^1,x^2)$ for $i=0,1,2,$ i.e. the pseudogroup of point transformations. Note that the 15 coefficients of such pairs of ODEs encode the so-called \emph{Thomas symbols} of the projective structure, defined as $\Pi^i_{jk}=\Gamma^i_{jk}-\tfrac14(\Gamma^l_{lk}\delta^i_j+\Gamma^l_{lj}\delta^i_k).$ More generally, point equivalence classes of pairs of ODEs $\ddot{x}^i=F^i(x^0,x^1,x^2,\dot{x}^1,\dot x^2)$ satisfying $F^i_{jkl}-\tfrac{3}{4}F^r_{r(jk}\delta^i_{l)}=0,$ where $F^i_{jkl}=\partial_{p^j}\partial_{p^k}\partial_{p^l}F^i(x^0,x^1,x^2,p^1,p^2),$ are in one-to-one correspondence with projective structures and can always be expressed as \eqref{eq:pair-2nd-ODEs}.

\subsection{Almost para-CR structure on the correspondence space}
\label{sec:corr-space-5d}
Now we  review a well-known correspondence space construction for (oriented) projective structures,  restricting ourselves to dimension 3.  This construction was first described in \cite{takeuchi}; see also \cite{Cap-corr}.

Firstly, recall from Definition \ref{def:cr-para-cr} and Remark \ref{rmk:cr-geom-remarks} that an almost para-CR structure is defined as a splitting of the contact distribution $\scF=\scV\oplus\scE$ where $\scV$ and $\scE$ are Lagrangian with respect  to the induced symplectic structure on the contact distribution.
\begin{proposition}
  Given a projective structure $[\nabla]$ on a 3-manifold $M,$ the  projectivized cotangent bundle, $\PP T^*M,$  is equipped with an  almost para-CR structure  $\scF=\scV\oplus\scE\subset T\PP T^*M$ where $\scF$ is the canonical contact distribution on $\PP T^*M,$ $\scV$ is the vertical distribution for the projection $\PP T^*M\to M$ and $\scE$ is obtained from the horizontal distributions defined by connections in $[\nabla].$
\end{proposition}
The same result holds for an oriented projective structure on the oriented projectivized cotangent bundle $\PP_+T^*M\to M.$ To describe the distribution $\scE\subset T\PP T^*M,$ note that each representative $\nabla\in [\nabla]$ defines a connection on $T^*M$ whose rank 3 horizontal distribution is denoted by $\what\scE\subset T T^*M$.  Descending to $\PP T^*M,$ one obtains a splitting $T\PP T^*M=\check\scE\oplus \scV,$ where  $\scV$ is the vertical distribution of $\PP T^*M\to M$ and the rank 3 distribution  $\check\scE$ is the push-forward of  $\what\scE.$ By the projective equivalence relation  \eqref{eq:proj-equiv} for two connections $\tilde\nabla,\nabla\in[\nabla],$ at each point    $(x;[p])\in\PP T^*M$ one has $\tilde\nabla_{v}p=\nabla_{v}p-\Upsilon(v)p-p(v)\Upsilon,$ for some  $\Upsilon\in T^*M.$ Thus, the rank 2 distribution $\scE\subset T\PP T^*M$ defined pointwise as $\scE_{(x;[p])}=\check\scE_{(x;[p])}\cap\Ker(p)$ is invariantly defined for the projective structure $[\nabla].$ Furthermore,   $\scV$ and $\scE$ are  Lagrangian with respect to the conformal symplectic structure on $\scF.$

Recall that $(2n-1)$-dimensional para-CR geometries are modeled by the flag variety of pairs of incident lines $\ell$ and hyperplanes $H$, i.e. 
\begin{equation}\label{eq:flat-path-geom}
 G\slash P_{1,n}  \cong \FF_{1,n}(\RR^{n+1}):=\{(\ell,H)\ \vline\ \ell\subset H\subset\RR^{n+1}\}\subset  \PP^{n}\times(\PP^{n})^*,
  \end{equation}
where   $G=\mathrm{SL}(n+1,\RR)$ and $P_{1,n}$ is the parabolic subgroup that stabilizes a chosen flag. The homogeneous space $\FF_{1,n}(\RR^{n+1})$ can be identified with $\PP T^*\PP^{n},$  i.e.  the projectivized cotangent bundle of the projective space $\PP^{n},$ equipped with its canonical contact structure and  Legendrian foliations given by the fibers of the  natural fibrations  $\PP T^*\PP^{n}\to\PP^{n}$ and $\PP T^*\PP^{n}\to(\PP^{n})^*$.

Thus, almost para-CR structures in dimension 5 are Cartan geometries of type $(\mathrm{SL}(4,\RR),P_{1,3})$ and are among the so-called \emph{regular} and \emph{normal} parabolic geometries, where $P_{1,3}=P_0\ltimes P_+$ in which  $P_+$ is the 5-dimensional nilradical and $P_0=\RR^\times\times \mathrm{GL}(2,\RR)\subset \mathrm{GL}(3,\RR).$ In the case of oriented almost para-CR structures that correspond to oriented projective structures, one has $P_0=\RR^+\times \mathrm{GL}^+(2,\RR).$  The harmonic curvature of 5-dimensional almost para-CR structures has  three components, also known as the fundamental invariants, and can be written as \[K_H=T^1+T^2+W.\]
The vanishing of $K_H$ implies local flatness, i.e. local equivalence to the flat model described previously. The vanishing of torsion components $T^1:\biw^2\scE\to\scV$ and $T^2:\biw^2\scV\to\scE$   is equivalent  to the integrability of the rank 2 distributions $\scE$ and $\scV,$ respectively.  Moreover, the vanishing of $W$ is equivalent to the existence of a 4-parameter family of Legendrian surfaces that are null with respect to the indefinite subconformal structure on the contact distribution, e.g. see \cite{KM-subconformal} for a review. An almost para-CR structure of dimension $2n-1$ arises from a projective structure on an $n$-dimensional manifold in the manner described above, if and only if $W=0$ and $T^2=0,$ e.g. see \cite{Cap-corr},

Starting from an oriented projective structure and its  Cartan connection \eqref{eq:3Dproj-CartanCurv}, we can express the $\mathfrak{sl}(4,\RR)$-valued Cartan connection of the induced 5-dimensional almost para-CR structure on $\PP_+ T^*M$ as
\begin{equation}\label{eq:para-CR-CartanCurv}
  \def\arraystretch{1.2}
\varphi=  \begin{pmatrix}
    -\eta^i_i & \mu_2 & \mu_1 &\mu_0\\
    \tau^2 & \eta^2_2 & \eta^2_1 & \mu_4\\
    \tau^1 & \eta^1_2 &  \eta^1_1 & \mu_3\\
    \tau^0 & \tau^4 &  \tau^3 & \eta^0_0
  \end{pmatrix} ,\quad    %
  \bvarphi=\exd\varphi+\varphi\w\varphi=
   \begin{pmatrix}
     0 & \bmu_2 & \bmu_1 & \bmu_0 \\
     0 & \beeta^2_2 & \beeta^2_1 & \bmu_4\\
     0 & \beeta^1_2 &  \beeta^1_1 & \bmu_3\\
     0 & \btau^4 &  \btau^3 & \beeta^0_0
   \end{pmatrix},
\end{equation}
where, comparing with \eqref{eq:3Dproj-CartanCurv}, we have
\[\tau^3=\eta^0_1,\quad  \tau^4=\eta^0_2,\quad \mu^3=\eta^1_0,\quad \mu^4=\eta^2_0.  \]
Using the projection $\what\pi\colon\cP\to\PP_+ T^*M,$ it follows that the  splitting $\scF=\scV\oplus\scE$ is given by
\[\scF=\what\pi_*\Ker\{\tau^0\},\quad  \scE=\what\pi_*\Ker\{\tau^0,\tau^3,\tau^4\},\quad \scV=\what\pi_*\Ker\{\tau^0,\tau^1,\tau^2\},\]  which is clearly Lagrangian since $\exd\tau^0\equiv \tau^1\w\tau^3+\tau^2\w\tau^4,$
mod $\{\tau^0\}.$  Taking any section $s\colon \PP_+ T^*M\to \cP,$ the torsion $T^1\in\Gamma(\scV\otimes \biw^2\scE^*)$ is given by
\begin{equation}\label{eq:torion-T2}
  T^1=(W^0_{112}\partial_{s^*\tau^3}+W^0_{212}\partial_{s^*\tau^4})\otimes(s^*\tau^1\w s^*\tau^2),
  \end{equation}
where,   using  \eqref{eq:proj-weyl} and working modulo $\{\tau^0\}$, the 2-forms entries $\btau^3$ and $\btau^4$ in $\bvarphi$ satisfy
\begin{equation}\label{eq:entries-of-T2}
  \btau^3\equiv W^0_{112}\tau^1\w\tau^2,\quad \btau^4\equiv W^0_{212}\tau^1\w\tau^2.
\end{equation}

From the description of a projective structure as the point equivalence class of a pair of second order ODEs  \eqref{eq:pair-2nd-ODEs}, one can obtain a local coordinate description of $\scE$ and $\scV.$ Taking $(x^0,x^1,x^0)$ as local coordinates on $M$ and $(p^1,p^2)$ as fiber coordinates for $\PP T^*M\to M,$ the description above gives  
\begin{equation}\label{eq:D1-D2-local}
\scV=\langle\partial_{p^1},\partial_{p^2} \rangle,\qquad \scE=\langle v_1,v_2\rangle,
\end{equation}
where
\begin{equation}\label{eq:basis-para-CR-D2-local}
  \begin{aligned} v_1&=\partial_{x^1}+p^1\partial_{x^0}+\left[-\Gamma^1_{00}(p^1)^3-\Gamma^2_{00}(p^1)^2p^2-\Gamma^0_{00}(p^1)^2+\Gamma^1_{11}p^1+\Gamma^2_{11}p^2+\Gamma^0_{11}\right] \partial_{p^1}\\
    & +\left[-\Gamma^1_{00}(p^1)^2  p^2 -  \Gamma^2_{00}p^1  (p^2)^2+ \Gamma^1_{20}(p^1)^2 -  (\Gamma^0_{00}+  \Gamma^1_{10}-  \Gamma^2_{20})p^1  p^2- \Gamma^2_{10}(p^2)^2\right.\\
    &+  \left. (\Gamma^0_{20}+  \Gamma^1_{21})p^1-  (\Gamma^0_{10}+  \Gamma^2_{21})p^2+\Gamma^0_{21}\right]\partial_{p^2},\\
  v_2&=\partial_{x^2}+p^2\partial_{x^0}+\left[-\Gamma^2_{00}(p^2)^3-\Gamma^1_{00}(p^2)^2p^1-\Gamma^0_{00}(p^2)^2+\Gamma^1_{22}p^1+\Gamma^2_{22}p^2+\Gamma^0_{22}\right] \partial_{p^2}\\
    & +\left[-\Gamma^1_{00}(p^1)^2  p^2 -  \Gamma^2_{00}p^1  (p^2)^2- \Gamma^1_{20}(p^1)^2 +  (-\Gamma^0_{00}+  \Gamma^1_{10}-  \Gamma^2_{20})p^1  (p^2)+ \Gamma^2_{10}(p^2)^2\right.\\
    &+  \left. (-\Gamma^0_{20}+  \Gamma^1_{21})p^1+  (\Gamma^0_{10}+  \Gamma^2_{21})p^2+\Gamma^0_{21}\right]\partial_{p^1}.
  \end{aligned}
\end{equation}
In this local coordinate system  the torsion $T^1$ can be explicitly expressed in terms of the first jet of the Thomas symbols, $\Pi^i_{jk}$, which will not be of importance to us.

\subsection{The twistor bundle and its 2-nondegenerate CR structure} 
\label{sec:twistor-bundle-2}

Given an oriented projective structure $[\nabla]$ on $M,$ let $(\PP_+ T^*M;\scE,\scV)$ be its corresponding almost para-CR structure on the oriented projectivized cotangent bundle, $\PP_+ T^*M,$ with co-oriented contact distribution $\scF=\scE\oplus\scV$. We define the twistor bundle of $[\nabla]$, denoted by $\nu\colon T^{[\nabla]}\to\PP_+ T^*M,$ as the bundle of almost complex structures on   $\scF$ that are compatible with its conformal  symplectic 2-form and  indefinite bilinear form. As a result, at every $p\in \PP_+T^*M,$ any such compatible almost complex structure acts on $\scF_p$ via the action of the structure group. Recall from \ref{sec:corr-space-5d} that almost para-CR structures on $\PP_+T^*M$   are Cartan geometries of type $(\mathrm{SL}(4,\RR),P_{1,3}),$ where $P_{1,3}=P_0\ltimes P_+$ and $P_0=\RR^+\times \mathrm{GL}^+(2,\RR)$ is referred to as the structure group.   At every $p\in \PP_+ T^*M,$ one finds that the almost complex structure
\[J=
  \begin{pmatrix}
    0 & -1\\
    1 & 0
  \end{pmatrix}
\]
 is  compatible  since $J\in P_0.$ Clearly $P_0$ acts transitively on   compatible almost complex structures on $\scF_p$ and the stabilizer of $J$ is $\RR^+\times \mathrm{CO}(2,\RR).$ As a result,  the fibers of the twistor bundle $\nu\colon\cTT\to \PP_+ T^*M$ are
\[\cTT_{p}:=\nu^{-1}(p)=\mathrm{SL}(2,\RR)\slash \mathrm{SO}(2,\RR)\cong\DD^2.\]
The stabilizer of $J$ defines the subgroup $P_J\subset P_{1,3}$ given by $P_J=(\RR^\times\times \mathrm{CO}(2,\RR))\ltimes P_+.$ Hence,      $\check\pi\colon\cP\to\cTT$ is a principal $P_J$-bundle and one has 
\begin{equation}\label{eq:twistor-bundle-3D-proj}
  \cTT\cong\cP\slash P_J=\cP\times_{P_{1,3}}P_{1,3}\slash P_J.
\end{equation}

Similarly to the two descriptions of $\cT$ in \ref{sec:lebr-feff-constr}, one can give another description of $\cTT$ using complexified Lagrangian null planes.  To see this, by \ref{sec:corr-space-5d}, one knows that in the almost para-CR structure on $\PP_+T^*M,$ the subconformal structure $[g]\subset\mathrm{Sym}^2(\scF^*)$ and the conformal  symplectic 2-form on the contact distribution $[\rho]\subset\biw^2\scF^*$ are given by
\[g=\tau^1\tau^3+\tau^2\tau^4,\quad \rho=\tau^1\w\tau^3+\tau^2\w\tau^4.\]
At each point $p\in \PP_+ T^*M$ the 1-parameter family  of (real or complex) null and Lagrangian planes in $\scF_p\subset T_p\PP_+ T^*M$, which as in   \eqref{eq:alpha-beta-planes} are called     $\alpha$-planes,  is 
\begin{equation}\label{eq:alpha-planes-nabla}
  \cA_p=\left\{\Ker\left\{\tau^0,a_1\tau^1-a_2\tau^2,a_2\tau^3+a_1\tau^4\right\}\ |\ [a_1:a_2]\in \PP^1\right\}.
  \end{equation}
Allowing $[a_1:a_2]$ to be in $\RR\PP^1$ or $\CC\PP^1,$ one obtains   real and complex $\alpha$-planes at $p\in M$, denoted by $\cA_p^\RR$ and $\cA_p^\CC,$ respectively. The set of  complex $\alpha$-planes without any real $\alpha$-plane i.e. $\cA^\CC_p\backslash\cA^\RR_p,$ is isomorphic to $\CC\PP^1\backslash\RR\PP^1.$ Similarly to our discussion in \ref{sec:lebr-feff-constr}, after fixing an orientation, the fibers of    $\cTT$ can be identified with a connected component of $\cA^\CC_p\backslash\cA^\RR_p$ which is $\DD^2.$  In other words,  at every  $p\in\PP_+ T^*M,$ using the action of $\mathrm{GL}(2,\RR)$ on $\scE_p$ and $\scV_p,$ one can establish a bijection between elements  of $\cA_p^\CC\backslash\cA_p^\RR$ and  the dual $\ri$-eigenspace for each compatible almost complex structure acting on  $\scE_p$ and $\scV_p.$  
\begin{proposition}\label{prop:twistor-bundle-2nondeg-CR}
  The 7-dimensional twistor bundle $\nu\colon\cTT\to\PP_+ T^*M$ of an oriented  projective structure $[\nabla]$ on a 3-manifold $M$ is equipped with a 2-nondegenerate CR structure of hypersurface type. The corank one distribution $\what\scF\subset T\cTT$ is the pre-image of the canonical contact distribution  $\scF\subset T\PP_+ T^*M,$ the Levi kernel $K\subset \what\scF$  is the vertical distribution of the $\DD^2$-bundle $\cTT\to\PP_+ T^*M$ with its canonical complex structure $\scK\oplus\cscK=\CC\otimes K$. With respect to the projection $\check\pi\colon\cP\to\cTT$ and  Cartan connection \eqref{eq:para-CR-CartanCurv}, at each point $z\in\cTT,$ the holomorphic distribution $\what\scH_z\subset \CC\otimes\what\scF_z$ is   $\check\pi_*\Ker\{\tau^0,\zeta^1,\zeta^2,\zeta^3\}$  where
\begin{equation}\label{eq:zetas-proj-str}
 \zeta^1=\tau^1+\ri\tau^2,\quad \zeta^2=\tau^3+\ri\tau^4,\quad \zeta^3=\eta^2_2-\eta^1_1+\ri(\eta^1_2+\eta^2_1).
  \end{equation}
\end{proposition}
\begin{proof}
The proof is similar to that of  Proposition \ref{prop:algebr-type-CR-twist}. From  \eqref{eq:twistor-bundle-3D-proj} it follows  that for any section $s\colon \cTT\to\cP$ the 1-forms $(s^*\tau^0,\cdots,s^*\tau^4,s^*(\eta^1_1-\eta^2_2),s^*(\eta^1_2+\eta^2_1))$ are semi-basic for $\cTT\to\PP_+T^*M.$ Using the projection $\nu\colon\cTT\to \PP_+ T^*M,$ define the corank 1 distribution
\[\what \scF=\nu_{*}^{-1}(\scF)=\check\pi_*\Ker\{\tau^0\}\subset T\cTT.\]
By abuse of notation we drop the pull-back $s^*$ in what follows. Consider the complex Pfaffian system
\begin{equation}\label{eq:holomorphic-2-nondeg}
  \cI=\{\tau^0,\zeta^1,\zeta^2,\zeta^3\},
  \end{equation}
where $\zeta^i$'s are as in \eqref{eq:zetas-proj-str}.  Using the structure equations \eqref{eq:para-CR-CartanCurv} and \eqref{eq:proj-weyl}, it is straightforward to show that $\cI$ is  integrable and defines a complex structure on $\widehat\scF$ whose holomorphic distribution is given by $\what\scH=\Ker\cI\subset \CC\otimes\what\scF$.

  Furthermore, by construction,  the Levi bracket on $\widehat\scF$ is degenerate since
  \begin{equation}\label{eq:contact-dist-on-cTT}
    \exd\tau^0\equiv \half (\zeta^1\w\czeta^2-\zeta^2\w\czeta^1),\mod\{\tau^0\}.
      \end{equation}
The complexification of degenerate directions of the Levi bracket is  $\scK\oplus \bar \scK\subset \CC\otimes K$ where $\scK=\langle\check\pi_*\partial_{\zeta^3}\rangle.$
Using the symbol algebra of the almost para-CR structure one obtains 
\begin{equation}\label{eq:2-nondege-CR}
  \exd \czeta^1\equiv \half \zeta^1\w\czeta^3,\quad
  \exd \czeta^2\equiv -\half \zeta^2\w\czeta^3,\mod \{\tau^0,\czeta^1,\czeta^2,\czeta^3\}. 
\end{equation}
By the definition of the higher Levi bracket  \eqref{eq:2nondeg-Levi-bracket},  relations \eqref{eq:2-nondege-CR} imply 2-nondegeneracy.
\end{proof}
Alternatively,  at every $z\in\cTT,$ the 1-forms $\zeta^1,\zeta^2$ in \eqref{eq:zetas-proj-str} can be viewed as the horizontal lift of the dual $\ri$-eigenspace of the almost complex structure that corresponds to $z$ with respect to some almost para-CR Weyl structure. The 1-form $\zeta^3$ is the canonical holomorphic 1-form on $\mathrm{SL}(2,\RR)\slash\mathrm{SO}(2,\RR)$ that is invariant under the action of $\mathrm{SO}(2,\RR).$ As before, this alternative construction of the CR structure is  independent of the choice of the Weyl structure for the almost para-CR structure.
 
\begin{remark}
    The induced CR structure on  $\cTT$ is a special case of a more general construction that has been studied in \cite{KM-subconformal} in the case of zero-curvature subconformal contact 5-manifolds and  in \cite{Gregorovic-CR} for a larger class of parabolic geometries. 
\end{remark}

\section{Weyl metrizability and CR submanifolds} 
\label{sec:weyl-metr-as}
In this section we show that the Weyl metrizability of an oriented projective structure is equivalent to the existence of a twistor CR manifold with a canonical choice of CR Weyl connection as a CR submanifold in the 7-dimensional twistor bundle of a projective structure. In Corollary \ref{cor:CR-type-N} we consider a wider class of nondegenerate CR submanifolds in the twistor bundle of the projective structure characterized by being transverse to the Levi kernel and  their holomorphic binary quartic being zero or having  a repeated root of multiplicity four.

\subsection{Projective class of a conformal Weyl structure} 
\label{sec:proj-class-weyl}
Here we obtain  necessary conditions for the Weyl metrizability of an oriented projective structure.  The Cartan geometry of a conformal structure $[g],$ is denoted by $(\cC\to M,\psi),$ satisfying \eqref{eq:conf-cartan-conn} and \eqref{eq:Cartan-curvature-conformal-conn}. The Cartan geometry of a projective structure, $[\nabla],$ is denoted as $(\cP\to M,\varphi),$ satisfying \eqref{eq:3Dproj-CartanCurv}.

Given a conformal structure $[g],$ let $\nabla^{[g]}$ be a conformal connection (Weyl structure)  associated to $s\colon\cC_0\to\cC.$  As a result, the pull-back of the 1-forms $\xi_i$ in \eqref{eq:conf-cartan-conn} satisfy
\begin{equation}\label{eq:rho-tensor-conformal}
s^*\xi_{i}=\sP_{ij}s^*\omega^j,  
\end{equation}
where the functions $\sP_{ij}$ on $\cC_0$ are referred to as the Schouten tensor or the Rho tensor of $\nabla^{[g]}.$  Subsequently, the inclusion $\cC_0\subset \cP$ gives  a $\mathrm{CO}(p,q)$ reduction of the projective structure of $[\nabla^{[g]}]$, i.e. a reduction of the Cartan bundle $\cP\to M,$ to the principal $\mathrm{CO}(p,q)$-bundle $\cC_0\to M$ defined by the conformal connection $\nabla^{[g]}.$  Straightforward computation shows, c.f.  the proof of Theorem \ref{thm:conf-beltrami}, that the pull-back of the Cartan connection  $\varphi$  for  the projective structure $[\nabla^{[g]}]$ to $\cC_0,$ in  matrix form \eqref{eq:para-CR-CartanCurv},   is given by
\begin{equation}\label{eq:Weyl-conformal-projective}
  \def\arraystretch{1.3}
  \begin{pmatrix}
    -\tfrac 34\theta^0  & \sQ_{2i}\omega^i & \sQ_{1i}\omega^i & \sQ_{0i}\omega^i\\
    \omega^2 & \tfrac 14\theta^0 & -\theta^3 & -\ve\theta^2\\
    \omega^1 & \theta^3 & \tfrac 14\theta^0  &-\ve\theta^1\\
    \omega^0 & \theta^2 & \theta^1 & \tfrac 14\theta^0
  \end{pmatrix}
\end{equation}
where, defining $[\ve_{ij}]$ such that $\ve_{00}=1,\ve_{11}=\ve_{22}=\ve$ and $\ve_{ij}=0$ for $i\neq j,$ one has
\begin{equation}\label{eq:Q-P-3}
  \sQ_{ij}=\half\ve^{kl}\sP_{kl}\ve_{ij}+\tfrac 18 (5\sP_{ij}-\sP_{ji}).
  \end{equation}
Let $\RT$ denote the bundle of  rays of future time-like  1-forms in the Lorentzian signature and   the  oriented projectivized cotangent bundle  in the Riemannian signature. The following holds.
\begin{proposition}\label{prop:necessary-cond-weyl-metr}
  If an oriented projective structure $[\nabla]$ on a 3-manifold $M$ is Weyl metrizable for a conformal structure $[g]$ then  there exists   a section $s\colon \RT \to\cTT$   with the property that $s(\RT)$ is a CR submanifold of $\cTT$ with an induced nondegenerate Levi bracket and of hypersurface type whose fundamental binary quartic is either zero or has a repeated root of multiplicity  four. Moreover, in terms of the CR coframing \eqref{eq:zetas-proj-str}, on $N$ one has the nonintegrability condition
  \begin{equation}\label{eq:nondege-conformal}
    \exd\zeta^1\w\zeta^1\w\czeta^1=\ve\tau^0\w\zeta^1\w\czeta^1\w\zeta^2,
  \end{equation}
 where $\ve=1$ or $-1$ if $[g]$ is Riemannian or Lorentzian, respectively.  
\end{proposition}
In the statement of the proposition, since $\RT\subseteq\PP_+T^*M$ is an open subset, by restriction, we have the fibration  $\cTT\to\RT$ using which we can define sections $s\colon\RT\to\cTT.$
\begin{proof}
  Given a conformal connection  $\nabla^{[g]}$ for $[g],$ by Proposition \ref{prop:twistor-bundle-2nondeg-CR},  the 2-nondegenerate CR structure on $\cTTg$ has holomorphic distribution   $\what\scH\subset \CC\otimes T\cTTg$  given by $\Ker\{\tau^0,\zeta^1,\zeta^2,\zeta^3\}$.
  Since the conformal connection  $\nabla^{[g]}$ corresponds to a section $s\colon\cC_0\to \cC$,  the projective connection $\nabla^{[g]}\in [\nabla^{[g]}]$ in turn defines an inclusion  of principal bundles
  \begin{equation}\label{eq:inclusion-Weyl-iota}
    \iota\colon \cC_0\to \cP,
      \end{equation}
      where $(\cP\to M,\varphi)$ is the Cartan geometric data for  the oriented projective structure $[\nabla^{[g]}].$

      By the construction of $\cT,$ it follows that $\cC_0\to\cT$ is a principal  $\mathrm{CO}(2,\RR)$-bundle. Hence, using \eqref{eq:inclusion-Weyl-iota}, one obtains a local embedding
      \begin{equation}\label{eq:Weyl-conn-twistor-embedding}
        \iota(\cT):=\iota(\cC_0)\slash\mathrm{CO}(2,\RR)\subset \cTTg\cong\cP\slash P_J,
      \end{equation}
wherein we are using the induced injection $\mathrm{CO}(2,\RR)\to P_J;$ see \ref{sec:twistor-bundle-2}.   Recalling that $\cT\cong\RT,$  \eqref{eq:Weyl-conn-twistor-embedding}   defines a section $\RT \to \cTTg.$ 
  
The pull-back of $\varphi$ for $[\nabla^{[g]}]$ to $\RT\subseteq\PP_+ T^*M$ is expressed as \eqref{eq:Weyl-conformal-projective}. Using the definition of the contact 1-form $\tau^0$ and the holomorphic 1-forms $\zeta^i$ from \eqref{eq:zetas-proj-str}, one obtains
\begin{equation}\label{eq:pull-back-CR-submanifold-cT}
  \iota^*\tau^0=\omega^0,\quad \iota^*\zeta^1=\omega^1+\ri\omega^2,\quad  \iota^*\zeta^2=\theta^1+\ri\theta^2,\quad  \iota^*\zeta^3=0,
  \end{equation}
which, by \eqref{eq:holom-1-forms-CR-cT},  coincides with the contact and holomorphic 1-forms for the CR structure on $\cT.$ In other words, $\iota(\cT)$ is a Levi-nondegenerate hypersurface-type CR submanifold of $\cTTg$ that is transverse to its Levi kernel.

From Proposition \ref{prop:algebr-type-CR-twist} we know that the fundamental binary quartic of the CR structure on $\cT$ is either zero or has a repeated root of multiplicity four. Lastly, checking  condition \eqref{eq:nondege-conformal} is immediate from \eqref{eq:pull-back-CR-submanifold-cT} and the structure equations  \eqref{eq:Cartan-curvature-conformal-conn} for conformal structures. 
\end{proof}
\begin{remark}\label{rmk:closed-weyl-structure}
  A Weyl structure $\nabla^{[g]}$ for which $\sP_{ij}=\sP_{ji},$ is called a \emph{closed Weyl structure} since, by structure equations \eqref{eq:cotton-york},   $\exd\theta^0=0$ holds.  Furthermore, given any section $s_0\colon M\to \cC_0,$  one has   $\nabla^{[g]}s_0^*g=-2s_0^*\theta^0\otimes s_0^*g$ for $g\in \mathrm{Sym}^2(T^*\cC)$   in \eqref{eq:conformal-metric}. Denoting $g_0=s_0^* g$ and $\phi_0=s^*_0\theta^0,$ from $\exd\phi_0=0,$ locally,  $\phi_0=\exd f$ holds for some function $f.$ It is then elementary to show that $\nabla^{[g]}g_1=0$ where $g_1=e^{2f} g_0$ and that $\nabla^{[g]}$ is the Levi-Civita connection of  $g_1.$ In other words, closed Weyl structures are locally the Levi-Civita connection of   representatives in $[g],$ defined up to positive homotheties. The Rho tensor $\sP_{ij}$ and the Ricci  tensor ${R_{ij}}$ of the metric are related via $\sP_{ij}=\tfrac{-1}{n-2}({R}_{ij}-\tfrac{1}{2(n-1)}R\ve_{ij})$ where $R=\ve^{ij}{R}_{ij}$ and $n=\mathrm{dim} M$ 
\end{remark}
\subsection{CR submanifolds as holomorphic sections of the Levi kernel}  
\label{sec:sections-levi-kerner}
To show that the conditions in Proposition \ref{prop:necessary-cond-weyl-metr} are also sufficient, we need the following definition and an observation about Weyl structures for twistor CR manifolds wherein $\PP_+T^*M$ denotes the bundle of rays of 1-forms. We recall from \ref{sec:from-3d-projective} that the Cartan geometry of a projective structure $[\nabla]$ is denoted as $(\cP\to M,\varphi)$ has correspondence space $(\cP\to\PP_+T^*M,\varphi)$ equipped with an almost para-CR structure and a 7-dimensional twistor bundle $\cTT\to \PP_+ T^*M$ equipped with a 2-nondegenerate CR structure.
\begin{definition} \label{def:hol-section}
  Given an oriented 3-dimensional projective structure $[\nabla]$ with twistor bundle $\cTT\to\PP_+T^*M,$ a section $\tilde t\colon N\to\cTT$ for an open subset $N\subset \PP_+ T^*M$ is called holomorphic if  the complex Pfaffian system $s^*\cI=s^*\{\tau^0,\zeta^1,\zeta^2,\zeta^3\},$  as defined in \eqref{eq:zetas-proj-str},   is integrable, where $s=\tilde s\circ\tilde t\colon N\to \cP,$ for some section $\tilde s\colon \cTT\to\cP$. 
\end{definition}
In Remark \ref{rmk:quartic-Lagrangian-fibration} we will comment on how \eqref{eq:hol-coframe-trans} implies that a holomorphic section, as defined above, is independent of the section $\tilde s\colon\cTT\to\cP.$
By the construction  in \ref{sec:twistor-bundle-2},  the fibers of $\cTT\to\PP_+ T^*M$ are the leaves of the Pfaffian system $\{\tau^0,\zeta^1,\czeta^1,\zeta^2,\czeta^2\}.$ It follows that, for any section $\tilde t\colon N\to \cTT,$ $N\subset\PP_+T^*M,$ the 1-forms $(s^*\tau^0,s^*\zeta^1,s^*\zeta^2,s^*\czeta^1,s^*\czeta^2)$ give a coframe on $N$ and 
\begin{equation}
  \label{eq:zeta3-holomor}
  s^*\zeta^3=a_0\tau^0+a_1\zeta^1+a_2\zeta^2+b_1\czeta^1+b_2\czeta^2,
  \end{equation}
  for some complex-valued functions $a_0,a_1,a_2,b_1,b_2$ on $N,$ where we have dropped $s^*$ on the right hand side. Since $\{\tau^0,\zeta^1,\zeta^2,\zeta^3\}$ is integrable on $\CC\otimes T^*\cTT,$ by Definition \ref{def:hol-section},  $\tilde t\colon N\to \cTT$   is holomorphic if and only if
  \begin{equation}
    \label{eq:b12-zero}
    b_1=b_2=0.
      \end{equation}

\begin{proposition}\label{prop:Weyl-structre-cr-submanifolds}
  Any holomorphic section $\tilde t\colon N\to \cTT$ for an open subset  $N\subset \PP_+ T^*M$ defines a nondegenerate CR structure of hypersurface type with Cartan geometric data $(\cQ\to N,\gamma),$ described in Theorem \ref{thm:CR-equiv-problem},  endowed with a  canonical CR connection $t\colon\cQ_0\to\cQ$ defined via the relation
  \[
    \begin{aligned}
      t^*\zeta^3&=a\tau^0,\\
      t^*\mu_2&\equiv 0 \mod\{\tau^0,\tau^1,\tau^2,\tau^4\},
          \end{aligned}
\]
  for some $a\in C^\infty(\cQ_0,\CC),$ where $\zeta^3$ is given in \eqref{eq:zetas-proj-str}.
\end{proposition}
\begin{proof}
  Recall from \ref{sec:corr-space-5d} that an oriented projective structure, $[\nabla],$ defines an almost para-CR structure on $\PP_+ T^*M,$ which is a Cartan geometry $(\cP\to \PP_+ T^*M,\varphi)$ of type $(\mathrm{SL}(4,\RR),P_{1,3}),$ where  $P_{1,3}=P_0\ltimes P_+$ and $P_0=\RR^+\times \mathrm{GL}^+(2,\RR).$ Since by \eqref{eq:twistor-bundle-3D-proj} the fibers of the twistor bundle $\cTT\to\PP_+ T^*M$ are $\mathrm{SL}(2,\RR)\slash\mathrm{SO}(2,\RR)\cong \DD^2,$ a section $\tilde t\colon N\to \cTT,N\subset \PP_+ T^*M,$ is equivalent to a  reduction of the structure group of the almost para-CR structure on $N$ from $P_0=\RR^+\times\mathrm{GL}^+(2,\RR)$ to $\widetilde P_0:=\RR^+\times\mathrm{CO}(2,\RR).$ As a result of such reduction, one obtains an inclusion $\iota_0\colon\widetilde\cP_0\to\cP$ where $\widetilde\cP_0$ is a principal $(\widetilde P_0\ltimes P_+)$-bundle. Moreover,  by  Definition \ref{def:hol-section} and the subsequent discussion, since $\tilde t\colon N\to\cTT$ is holomorphic, it follows from \eqref{eq:zeta3-holomor} and \eqref{eq:b12-zero} that
  \begin{equation}
    \label{eq:zeta3-red-holom}
    \tilde t^*\zeta^3=a_0\tau^0+a_1\zeta^1+a_2\zeta^2,
      \end{equation}
  for some $a_0,a_1,a_2\in C^\infty(\widetilde\cP_0,\CC).$ One can explicitly find the dependency of  $a_i$'s on   $\widetilde P_0\ltimes P_+$ as follows.  If  two points $p,q\in\widetilde\cP_0$ lie on the fiber over $x\in N\subset\PP_+ T^*M,$ then  one has $q=\bg^{-1}p,$ for  some $\widetilde P_0\ltimes P_+$-valued $\bg$ given as
  \begin{equation}
    \def\arraystretch{1}
\bg= \begin{pmatrix}
        \tfrac{1}{\bb_0^2\bb_1}  & 0 & 0 & 0 \\
        0 &\bb_0\cos(\bc)    & -\bb_0\sin(\bc)   &   0\\
        0 &\bb_0\cos(\bc)    & \bb_0\sin(\bc)   &   0\\
        0 &0    & 0   &  \bb_1\\
      \end{pmatrix}      \begin{pmatrix}
        1 & \bq_2 & \bq_1 & \bq_0+\half \bq_1\bq_3+\half \bq_2\bq_4 \\
        0 &1    & 0   &    \bq_4\\
        0 &0    & 1   &   \bq_3\\
        0 &0    & 0   &  1\\
      \end{pmatrix}.
  \end{equation}
  The transformation of the Cartan connection \eqref{eq:3Dproj-CartanCurv} along the fibers of $\widetilde\cP_0\to N$ are given by
  \[\varphi(p)\to \varphi(\bg^{-1}p)=\bg^{-1}\varphi\bg+\bg^{-1}\exd\bg,\]
  where $\bg\in C^{\infty}(\widetilde\cP_0, \widetilde P_0\ltimes P_+ ).$   It is straightforward to find  the transformations $\tau^0(p)\to\tau^0(\bg^{-1}p)$ and $\zeta^i(p)\to \zeta^i(\bg^{-1}p)$ to be 
  \begin{equation}\label{eq:hol-coframe-trans}
    \begin{gathered}
      \tau^0\to \tfrac{1}{\bb_0^2\bb_1^2}\tau^0,\quad \zeta^1\to \tfrac{e^{\ri \bc}}{\bb_0^3\bb_1}\zeta^1-\tfrac{\bq_{34}}{\bb_0^2\bb_1^2}\tau^0,\quad \zeta^2\to \tfrac{\bb_0e^{\ri \bc}}{\bb_1}\zeta^2+\tfrac{\bq_{12}}{\bb_0^2\bb_1^2}\tau^0,\\
      \zeta^3\to e^{2\ri\bc}\zeta^3+\tfrac{\bb_0e^{\ri \bc}}{\bb_1}\zeta^2+\tfrac{e^{\ri \bc}}{\bb_0^3\bb_1}\zeta^1-\tfrac{\bq_{12}\bq_{34}}{\bb_0^2\bb_1^2}\tau^0,\\
    \end{gathered}
  \end{equation}
  where $\bq_{12}=\bq_1+\ri \bq_2,$ $\bq_{34}=\bq_3+\ri \bq_4,$ and $e^{\ri \bc}=\cos(\bc)+\ri\sin(\bc).$ Subsequently, the transformation of
  $a_i$'s along the fibers of $\widetilde\cP_0$ is found to be
  \[
    \begin{aligned}
      a_0(\bg^{-1} p)=& \bb_0^2\bb_1^2e^{2\bc\ri}a_0(p)+{\bb_1}{\bb_0^3}e^{\bc\ri}\bq_{34}a_1(p)-\tfrac{\bb_1}{\bb_0}e^{\bc\ri}\bq_{12}a_{2}(p),\\      
      a_1(\bg^{-1}p)=& \bb_0^3\bb_1e^{\bc\ri}a_1(p)+\bq_{12},\\
      a_2(\bg^{-1} p)=& \bb_0^{-1}\bb_1e^{\bc\ri}a_2(p)-\bq_{34}.
    \end{aligned}
  \]
  At the infinitesimal level, the transformations above correspond to the differential relations
  \begin{equation}\label{eq:da_123}
        \begin{aligned}
      \exd a_0\equiv & 2a_0(\eta^1_1+\eta^3_3+\ri\eta^2_1)+a_1\mu_{34}-a_2\mu_{12},\\
      \exd a_1\equiv & a_1(3\eta^1_1+\eta^3_3+\ri\eta^2_1)+\mu_{12},\\
      \exd a_2\equiv & a_2(-\eta^1_1-\eta^3_3+\ri\eta^2_1)-\mu_{34}, 
    \end{aligned}
  \end{equation}
 mod $\{\tau^0,\cdots,\tau^4\},$ where $\mu_{12}=\mu_1+\ri\mu_2$ and $\mu_{34}=\mu_3+\ri\mu_4.$
  As a result, defining
\[\cP_1=\{ p\in\widetilde\cP_0\ \vline \ a_1(p)=a_2(p)=0\},\]
one obtains a principal $(\widetilde P_0\ltimes \RR)$-subbundle $\iota_1\colon \cP_1\subset \widetilde\cP_0$  over which $a_1=a_2=0$ and therefore
\[\iota^*_{01}\zeta^3=a_0\tau^0,\]
for some $a_0\in C^{\infty}(\cP_1,\CC)$ where $\iota_{01}=\iota_0\circ\iota_1\colon\cP_1\to\cP.$  Setting $a_1=a_2=0$ in differential relations \eqref{eq:da_123}, and using the complex coframe \eqref{eq:zetas-proj-str} one arrives at
\begin{equation}
  \label{eq:X-Ys}
    \begin{aligned}
        \iota_{01}^*\mu_{12}=&X_0\tau^0+X_1\zeta^1-Y_1\zeta^2-\half a_0\czeta^2 +\half (W^1_{212}+W^2_{121}+\ri W^1_{112}+\ri W^2_{221})\czeta^1,\\
    \iota_{01}^* \mu_{34}=& Y_0\tau^0+Y_1\zeta^1 +Y_2\zeta^2+\half a_0\czeta^1,
  \end{aligned}
\end{equation}
for some complex-valued functions $X_0,X_1,Y_0,Y_1,Y_2$ on $\cP_1.$ Note that to arrive at \eqref{eq:X-Ys} one needs to complete the expressions  \eqref{eq:da_123} by finding the differential relations among all coframe derivatives of $a_i$'s and $X_i$'s, $Y_i$'s and the projective Weyl curvature. This can be done by replacing the real and imaginary parts of $\zeta^3$ according to \eqref{eq:zeta3-red-holom} in the Cartan connection $\varphi$ of the projective structure and making sure the structure equations \eqref{eq:3Dproj-CartanCurv} remain valid. Carrying out the relevant computations, although tedious, is straightforward. It is rather remarkable that the expressions for $\exd a_1$ and $\exd a_2,$ once pulled-back to $\cP_1$ result in  simple relations \eqref{eq:X-Ys}.

Similarly, one can find the group action on $\Re Y_1$ to be
\[\Re Y_1(\bg^{-1}p)=\bb_0^2\bb_1^2\Re Y_1(p)+\bq_0\]
which infinitesimally corresponds to the differential relation
\[\exd\Re Y_1\equiv 2\Re Y_1(\eta^0_0+\eta^1_1)+\mu_0,\mod\{\tau^0,\tau^1,\tau^2,\tau^3,\tau^4\}.\]
Thus, one obtains a sub-bundle $\iota_2\colon\cP_2\to\cP_1,$ which is a principal $\mathrm{CO}(2,\RR)\times\RR^+$-bundle,  defined as
\begin{equation}\label{eq:P3-Y2zero}
  \cP_2=\left\{p\in\cP_1\ \vline\ \Re Y_1(p)=0\right\}.
  \end{equation}
Any holomorphic section $\tilde t\colon N\to \cTT, N\subset \PP_+ T^*M,$  defines a nondegenerate CR structure of hypersurface type on $N$ since $\scF:=\Ker\{\tau^0\}$ is the canonical contact distribution and $\scH:=\Ker\{\tilde t^*\cI\}\subset \CC\otimes \scF$ is an integrable holomorphic distribution of a complex structure which by \eqref{eq:contact-dist-on-cTT}  satisfies the compatibility condition.  In terms of the Cartan geometric description of Theorem \ref{thm:CR-equiv-problem}, for the induced CR structure on $N\subset\PP_+ T^*M,$  the  principal $\widetilde P_0$-bundle $\cP_2\to N$ defines a CR (Weyl) connection $t\colon \cQ_0\to \cQ$   on $N$. Setting $\iota_{012}:=\iota_0\circ\iota_1\circ\iota_2\colon\cP_2\to \cP,$  one can express the entries of  $\iota_{012}^*\gamma$ in \eqref{eq:5DCR-Cartan-conn} as
\[
  \begin{gathered}
    \alpha^0=-2\tau^0,\quad \alpha=\zeta^1,\quad \beta=\zeta^2,\quad \Re\phi^0_0=\eta^1_1+(\half\Re a_0)\tau^0,\quad \Re\phi^1_1=-2\eta^1_1+(\Re a_0)\tau^0\quad \phi^2_1=\half (\Re Y_3)\tau^0,\\
    \Im\phi^0_0=-\Im\phi^1_1=-\half\eta^1_2+(\tfrac 14\Im a_0- \tfrac16 \Im Y_2)\tau^0,\quad \phi^1_2=(\tfrac 34 W^1_{212}+\tfrac 34 W^2_{121}+\half \Re X_1)\tau^0-W^0_{212}\tau^1+W^0_{112}\tau^2.
  \end{gathered}
\]
 Using the entries above, one can proceed to find $\pi,\sigma$ and $\sigma_0$ in \eqref{eq:5DCR-Cartan-conn}  and compute the coefficients of the fundamental binary quartic to be
\begin{equation}\label{eq:C_is-hol-sec}
  C_0=0,\quad C_1=\Im Y_2,\quad C_2=\tfrac 43 \Im Y_1,\quad C_3=\tfrac 32(W^2_{221}-W^1_{112})-\Im X_1,\quad C_4=W^0_{112;1}+W^0_{212;2}
\end{equation}
\end{proof}
\begin{remark}\label{rmk:quartic-Lagrangian-fibration}
The expression for $C_4$ in \eqref{eq:C_is-hol-sec} is given in terms of the first jet the fundamental torsion \eqref{eq:torion-T2} for the almost para-CR structure on $\PP_+ T^*M$ defined by $[\nabla],$ which leads to Corollary \ref{cor:conf-rigid-flat}.  Furthermore, it follows from \eqref{eq:hol-coframe-trans} that the property of being holomorphic in Definition \ref{def:hol-section} for a section $\tilde t\colon N\to\cTT, N\subset \PP_+T^*M$ is invariant under the action of  $\tilde P_0\ltimes P_+$, and, therefore,  is independent of the section $\tilde s\colon\cTT\to\cP$.   Lastly, it is clear from  $C_0=0$ in  \eqref{eq:C_is-hol-sec}  that  the binary quartic \eqref{eq:quartic} has a real root at $[a_1:a_2]=[1:0]$ which, by the parametrization \eqref{eq:alpha-beta-planes} and Proposition \ref{prop:quartic-integrability}, corresponds to the  fibers of $\PP_+ T^*M\to M.$
\end{remark}
\subsection{CR characterization of Weyl metrizability}  
\label{sec:main-theorem} 
Now we can use Proposition \ref{prop:Weyl-structre-cr-submanifolds} to show that the necessary conditions in Proposition \ref{prop:necessary-cond-weyl-metr} are also sufficient. We first give a Cartan geometric proof of the following well-known lemma, e.g. see \cite{Lebrun-thesis, Mettler-3}, which shows  that given a projective structure $[\nabla],$  any conformal structure $[g]$ can have at most one conformal connection that satisfies $\nabla^{[g]}\in[\nabla]$.
\begin{lemma}\label{lem:proj-conf-intersect}
  If a projective structure $[\nabla]$ is Weyl metrizable with respect to a conformal structure $[g],$ then the property $\nabla^{[g]}\in [\nabla]$ determines a unique  conformal connection for $[g].$ 
\end{lemma}
\begin{proof}
  Let $\nabla^1$ be a conformal connection for $[g]$ associated to $s_1\colon\cC_0\to\cC,$ with respect to which 
  \[\exd\omega^i=-\omega^i_j\w\omega^j\]
  holds  where, in relation to \eqref{eq:conf-cartan-conn}, one has  $[\omega^i_j]\in\Omega^1(\cC_0,\mathfrak{co}(p,q)),p+q=3,$ and the only non-zero 1-form entries in $[\omega^i_j]$ are
  \[\omega^1_0=s_1^*\theta^1,\qquad \omega^2_0=s_1^*\theta^2,\qquad \omega^1_2=s_1^*\theta^3,\qquad \omega^0_0=s_1^*\theta^0.\]
  By  Definition \ref{def:weyl-str} for a general Weyl structure, any other conformal connection is related to $s_1$ by an action of the nilradical factor $Q_+\subset Q_1$ at each point, where  $Q_1\subset\mathrm{SO}(p+1,q+1)$ is the first parabolic subgroup, discussed in  \ref{sec:conf-geom-dimens}. In other words, if $s_2\colon\cC_0\to\cC$ is another conformal connection, then at every $p\in\cC_0,$ $s_2(p)=\bg_p^{-1}s_1(p),$ for some $\bg_p\in Q_+\subset Q_1,$ i.e.  
  \begin{equation}\label{eq:change-Weyl}
    \psi_2=\bg^{-1}\psi_1\bg+\bg^{-1}\exd\bg.
  \end{equation}
  where $\psi_i=s_i^*\psi$ and $\bg\in C^{\infty}(\cC,Q_+).$ Explicitly,  expressing elements of $Q_+\subset \mathrm{SO}(p+1,q+1)$ as
  \[\def\arraystretch{1}
\bg_p=      \begin{pmatrix}
        1 & q_2 & q_1 & q_0 & \half \ve_{ij}q^iq^j\\
        0 &1    & 0   &  0  & \ve q_2\\
        0 &0    & 1   &  0  & \ve q_1\\
        0 &0    & 0   &  1  & q_0\\
        0 &0    & 0   &  0  & 1\\
      \end{pmatrix},\]
    one obtains from \eqref{eq:change-Weyl}   that the  Weyl connection  1-forms $\omega^i_j$ transform to
    \begin{equation}\label{eq:conformal-related-connection}
      R^*_\bg\omega^i_j=\omega^i_j+q_j\omega^i-q_k\omega^l\ve_{lj}\ve^{ki}+q_k\omega^k\delta^i_j.
          \end{equation}
  where we have dropped pull-backs by $s_1.$

Carrying out the same steps in order to relate two projective connections in a projective structure $[\nabla]$ leads to  \eqref{eq:proj-equiv}, from which it follows that two conformal connections $\nabla^1$ and $\nabla^2$ are projectively equivalent if and only if their respective connection forms $\omega^i_j$ and $R^*_q\omega^i_j$ are related by
\begin{equation}\label{eq:projetive-conformal-related}
  R^*_\bg\omega^i_j-\omega^i_j=r_j\omega^i+\delta^i_jr_k\omega^k.
  \end{equation}
for some functions $r_0,r_1,r_2$ on $\cC_0.$ Thus, the projective equivalence of two conformal connections  $\nabla^1$ and $\nabla^2$ for $[g]$ implies
\begin{equation}\label{eq:q-r}
  q_j\omega^i-q_k\omega^l\ve_{lj}\ve^{ki}+q_k\omega^k\delta^i_j=r_j\omega^i+\delta^i_jr_k\omega^k.
  \end{equation}
Summing over $i,j,$ one obtains $r_k=\tfrac 3{4}q_k.$ Subsequent insertion into  \eqref{eq:q-r} and comparing both sides yields $q_0=q_1=q_2=0.$
\end{proof}
Now we can state the main theorem of this article. 
\begin{theorem}\label{thm:characterization-weyl-metr} 
  Given an oriented projective structure $[\nabla]$ on a 3-manifold $M,$   there is a one-to-one correspondence between conformal structures $[g]$ with respect to which $[\nabla]$ is Weyl metrizable and 5-dimensional Levi nondegenerate CR submanifolds of hypersurface type that are transverse to the Levi kernel of the 7-dimensional twistor bundle, $\cTT,$ whose fundamental binary quartic is either zero or has a repeated root of multiplicity  four   and satisfy  the non-integrability condition \eqref{eq:nondege-conformal} for $\ve= 1$ or $-1$ with respect to the canonical CR connection described in Proposition \ref{prop:Weyl-structre-cr-submanifolds}%
\end{theorem}

\begin{proof}
  By Proposition \ref{prop:necessary-cond-weyl-metr}, we only need to prove the sufficiency  part. By the proof of Proposition \ref{prop:Weyl-structre-cr-submanifolds}, a holomorphic section $\tilde t\colon N\to \cTT, N\subset \PP_+ T^*M,$  defines a nondegenerate CR structure on $N$ with a canonical CR connection whose fundamental binary quartic has coefficients \eqref{eq:C_is-hol-sec}, together with an inclusion $\iota_{012}\colon\cP_2\to\cP$ which is a $\mathrm{CO}(2)\times\RR^+$-reduction of the Cartan bundle for the projective structure $[\nabla]$. 

The condition \eqref{eq:nondege-conformal} for the  CR connection in Proposition \ref{prop:Weyl-structre-cr-submanifolds}  gives
  \begin{equation}\label{eq:necess-cond-1}
    \exd\zeta^1\w\zeta^1\w\czeta^1=Y_2\tau^0\w\zeta^1\w\czeta^1\w\zeta^2.
      \end{equation}
Thus, by \eqref{eq:necess-cond-1}, we assume $Y_2\neq 0.$  It is a straightforward calculation to find the  differential relations
\begin{equation}
  \label{eq:Y-X-Ws}
      \begin{aligned}
      \tfrac{\partial}{\partial \czeta^1}\Im Y_2-\tfrac{\partial}{\partial \czeta^2}\Im Y_1=&-\tfrac 34 \ri Y_0, \\
      \tfrac{\partial}{\partial\czeta^2}Y_0=&a_0Y_2,\\
      4\tfrac{\partial}{\partial\czeta^1}\Im Y_1+4\tfrac{\partial}{\partial\czeta^2}\Im X_1=&W^2_{220}+W^2_{101}-2W^1_{102}+W^0_{112}Y_2\\
      &+\ri(W^2_{120}-W^1_{202}+2W^2_{201}+W^0_{212}Y_2-3X_0).
        \end{aligned}
\end{equation}
where $Y_i$'s and $X_i$'s are the functions obtained from reductions in \eqref{eq:X-Ys} and $W^i_{jkl}$'s are the entries of the projective Weyl curvature.  The condition that the fundamental binary quartic  for the CR structure, \eqref{eq:W-CR-quartic}, has a repeated root of multiplicity  four implies
  \begin{equation}\label{eq:typeN}
    C_1=C_2=C_3=0,
      \end{equation}
Combining the relation in    \eqref{eq:typeN},  \eqref{eq:C_is-hol-sec} and \eqref{eq:Y-X-Ws}, together with $\Re Y_1=0$ from \eqref{eq:P3-Y2zero}, it follows that 
      \begin{equation}\label{eq:a0-Y012-X01}
                \begin{aligned}
          a_0=&Y_0=Y_1=\Im Y_2=0,\quad \Im X_1= \tfrac 32(W^2_{221}-W^1_{112}),\\
          3X_0=& 2W^0_{221}\Re Y_2-5W^2_{120}- W^1_{202}-4 W^2_{201}+ \ri\left( 2W^0_{112}\Re Y_2+5W^2_{220}-W^2_{101}-4 W^1_{102}\right),
        \end{aligned}
              \end{equation}
holds on $\cP_2.$      The differential consequences of the relations above are   straightforward to obtain and we will only provide the final result as  algebraic relations among the entries of $W^i_{jkl}$ and  $W_{ijk}$ in \eqref{eq:proj-weyl}:
      \[
        \begin{gathered}
         W^2_{001}= (W^1_{112}-3 W^2_{221})\Re Y_2,\qquad  W^1_{002}=(W^2_{221}-3W^1_{112})\Re Y_2 ,\qquad  W^2_{020}=(W^1_{212}-W^2_{121})\Re Y_2,\\
         W^2_{220}=-\half W^0_{112}\Re Y_2+\tfrac 32 W^1_{102},\qquad  W^2_{120}=-\half W^0_{212}\Re Y_2-\tfrac 32 W^2_{201},\qquad
         W^1_{202}=-W^0_{212}\Re Y_2-4W^2_{201}\\
          W^2_{101}=W^0_{112}\Re Y_2-4W^1_{102},\qquad W_{110}=\tfrac 16(2W^0_{112;2}+3W^0_{212;1})\Re Y_2+W^2_{121;0}+\tfrac 32 W^2_{201;1}+\tfrac 83 W^2_{220;2}.
        \end{gathered}
      \]
As a result,  the expression  for $X_0$ in ~\eqref{eq:a0-Y012-X01} simplifies to
      \[X_0=\tfrac 52W^2_{201}+\half W^0_{212}\Re Y_2+\ri( \tfrac 52 W^1_{102}-\half W^0_{112}\Re Y_2).\]
      Lastly,  using  $\exd^2\tau^1=\exd^2\tau^2=0,$ one obtains
      \begin{equation}
        \label{eq:d-RY2}
        \exd\Re Y_2=2\Re Y_2(\eta^0_0-\eta^1_1)+\Re Y_{2;0}\omega^0.
            \end{equation}
 Moreover, differential relations between coframe derivatives of $Y_i$'s yield  $\Re Y_{0;3}=\Re Y_{2;0}.$ Since    \eqref{eq:a0-Y012-X01} gives  $Y_0=0$, it follows from \eqref{eq:d-RY2} that $\exd\Re Y_2=2\Re Y_2(\eta^0_0-\eta^1_1),$ which means that  $\Re Y_2$ is a constant up to positive homotheties, i.e.
      \[\exd \Re Y_2=2(\eta^0_0-\eta^1_1)\Re Y_2\Longleftrightarrow \Re Y_2({\bg}^{-1} p)=\bb_1^{-2}\bb_0^2\Re Y_2(p). \] 
      Since by \eqref{eq:necess-cond-1} it is assumed that $Y_2\neq 0,$ there is a natural sub-bundle $\iota_3\colon\cP_3\to\cP_2$ defined as
\[\cP_3=\left\{p\in\cP_2\ \vline\ \Re Y_2(p)=-\ve\right\}\]
where $\ve$ can be $+1$ or $-1$ and \eqref{eq:necess-cond-1} takes the form \eqref{eq:nondege-conformal} since $\Im Y_2=0$ by \eqref{eq:a0-Y012-X01}. As a result, $\cP_3\to N$ is a   principal $\mathrm{CO}(2,\RR)$-bundle. Defining $\iota:=\iota_0\circ\iota_1\circ\iota_2\circ\iota_3\colon\cP_3\to\cP,$ one has
\begin{equation}\label{eq:reduced-projective-conn-conformal}
  \def\arraystretch{1.3}
\iota^*\varphi=  \begin{pmatrix}
    -3\eta_0^0  & \sQ_{2i}\tau^i & \sQ_{1i}\tau^i & \sQ_{0i}\tau^i\\
    \tau^2 & \eta_0^0 & -\eta^1_2 & -\ve\tau^4\\
    \tau^1 & \eta^1_2 & \eta_0^0  & -\ve\tau^3\\
    \tau^0 & \tau^4 & \tau^3 & \eta_0^0
  \end{pmatrix},
\end{equation}
which is obtained from the Cartan connection  $\varphi$ in  \eqref{eq:para-CR-CartanCurv} for the initial projective structure $[\nabla]$  by applying admissible gauge transformations, assuming \eqref{eq:typeN} and $Y_2\neq 0$, wherein
\[
  \begin{gathered}
        \sQ_{20}=\half W^0_{112}\ve+\tfrac 52 W^1_{102},\quad  \sQ_{21}=-W^1_{112}+2W^2_{221},\quad \sQ_{22}=\Re X_1+\half (W^1_{212}-W^2_{121}),\\
    \sQ_{10}=-\half W^0_{212}\ve +\tfrac 52 W^2_{201},\quad \sQ_{11}=\Re X_1+\half (W^2_{121}-W^1_{212}),\quad \sQ_{12}=2W^1_{112}-W^2_{221},\\
    \sQ_{00}=\ve\Re X_1+\half \ve(W^1_{212}+W^2_{121}),\quad  \sQ_{01}=-\half (W^0_{212}\ve+ W^2_{201}),\quad \sQ_{02}=\half (W^0_{112}\ve -W^1_{102}).
    \end{gathered}
  \]
  Comparing \eqref{eq:reduced-projective-conn-conformal} and  \eqref{eq:Weyl-conformal-projective}, it is clear that $\iota^*\varphi$ is a conformal connection for a conformal structure  on $M$ for which ~\eqref{eq:conformal-metric} is 
  \[g=(\tau^0)^2+\ve(\tau^1)^2+\ve(\tau^2)^2.\]

  By  the homogeneity of    $\iota^*\varphi$ in \eqref{eq:reduced-projective-conn-conformal} when pulled-back to the leaves of $\{\tau^0,\tau^1,\tau^2\},$ it is clear that the (open) fibers of  $\cP_3\to M,$ can be completed to   $\mathrm{CO}^o(p,q).$   Since the fibers of $\cP_3\to N, N\subset \PP_+ T^*M,$   are $\mathrm{CO}(2,\RR),$ it follows that  the (open) fibers of $N\to M$ can be completed to  $\mathrm{CO}^o(p,q)\slash\mathrm{CO}(2,\RR)$ as discussed in \ref{sec:lebr-feff-constr}, where  $\ve=1$ or $-1$ if  $(p,q)=(3,0)$ or $(2,1)$, respectively. 
\end{proof}

More generally, relaxing condition \eqref{eq:nondege-conformal} gives the following.

\begin{corollary}\label{cor:CR-type-N}
  Given an oriented projective structure $[\nabla]$ on a 3-manifold $M,$   suppose $N\subset \PP_+T^*M$ is a  Levi nondegenerate CR submanifold  of hypersurface type that is transverse to the Levi kernel of $\cTT$ whose fundamental binary quartic is either zero or has a repeated root of multiplicity. Then either $N$ corresponds to a conformal connection or  it  satisfies     \eqref{eq:nondege-conformal} for $\ve=0$ in which case $\cTT$  is locally  fibered over a complex curve and a distinguished  line field is induced on $M.$   Assuming analyticity, the local generality of such nondegenerate CR 5-manifolds for which $\ve=0$ is 4 functions of 3 variables. 
\end{corollary}
\begin{proof}
  By Theorem  \ref{thm:characterization-weyl-metr}, when $\ve=\pm 1$ the nondegenerate CR structure on  $N$ corresponds to a conformal connection. When $\ve=0,$ then, by \eqref{eq:nondege-conformal}, the Pfaffian system $\cF=\{\zeta^1,\czeta^1\}=\{\tau^1,\tau^2\}$ is integrable. The leaf space of the 3-dimensional leaves  of $\cF$ is a complex curve equipped with the coframing $(\zeta^1,\czeta^1).$  One can check that the Pfaffian system $\cF$ is well-defined on $M$ since it is preserved by the action of the fibers of $\pi_3\colon\cP_3\to M$ as defined in \eqref{eq:P3-Y2zero} i.e. for any vertical vector field $X$ in  $\pi_3:\cP_3\to M$ the Lie derivative of $\cF$ satisfies $\cL_X\cF=\cF.$ Thus,  $\Ker\cF$ defines a line field $\ell\subset TM$ where $\ell=\langle(\pi_3)_*\tfrac{\partial}{\partial\tau^0}\rangle.$ In fact, one has $\Ker\cF=\ell\oplus\scV$ where $\scV$ is the rank 2 vertical distribution for $N\to M.$ Carrying out the computations to find the resulting structure equations and algebraic relations among the components of the projective Weyl curvature, which are tedious but straightforward, goes similarly to that  of Theorem \ref{thm:characterization-weyl-metr} with the difference that in \eqref{eq:a0-Y012-X01} in general one does not have $a_0=0.$ We will not provide the details here.  Doing so will allow one to use Cartan-K\"ahler machinery and obtain the local generality of such CR structures to be 4 functions of 3 variables. 
\end{proof}

\begin{remark}
  The class of Levi nondegenerate CR structures in Corollary \ref{cor:CR-type-N}  satisfy \eqref{eq:nondege-conformal} for $\ve=1,-1$ or $0.$ If $\ve=\pm 1,$ then they are foliated by null Lagrangian surfaces and correspond to conformal connections and, by our discussion in \ref{sec:introduction}, their local generality depends on 5 functions of 3 variables. If $\ve=0,$ then they are locally foliated by 3-dimensional leaves.   One can analyze the structure of the 3-dimensional leaves of the latter class of CR structures. To do so,  one sets $\tau^1=\tau^2=0$ in the structure equations and obtains that  the 1-form $\tau^0$ is integrable on each leaf. The respective leaves of $\{\tau^0,\tau^1,\tau^2\}$ were discussed  in \ref{sec:cr-structures-type-begin} as holomorphic curves that are null and Legendrian, c.f. \cite{Bryant-HolomorphicCurves}.  Inline with the result of \cite{Lebrun-FoliatedCR} on foliated CR manifolds with compact complex manifolds, one can ask whether there is an example for which all such 3-dimensional leaves are compact. 
\end{remark}

\section{Conformal Beltrami theorem and the Einstein-Weyl condition}  \label{sec:two-corollaries}
This section is about two corollaries from  ~\ref{sec:weyl-metr-as}. The first one is the conformal Beltrami theorem and the second one is a CR characterization of the Einstein-Weyl condition. 
\subsection{Conformal rigidity of flat projective structures}
We first prove the following corollary which can be taken as  a conformal analogue of Beltrami's theorem in dimension three.
\begin{corollary}\label{cor:conf-rigid-flat}
A 3-dimensional conformal structure  locally has a projectively flat conformal connection if and only if it is locally conformally flat. 
\end{corollary}
\begin{proof}
From  \eqref{eq:C_is-hol-sec} it follows  that if the projective structure is flat, then the CR structure defined by a holomorphic section  transverse to the Levi kernel of  $\cTT$ satisfies $C_0=C_4=0,$ i.e.  the holomorphic binary quartic \eqref{eq:W-CR-quartic} has a root at infinity and at zero. However, by  Theorem \ref{thm:characterization-weyl-metr}, for holomorphic sections that are defined by  conformal connections the binary quartic is either zero or has a repeated root of multiplicity four. Thus,  the binary quartic of a projectively flat conformal connection is identically zero which, by Proposition \ref{prop:algebr-type-CR-twist}, implies conformal flatness. 

  Alternatively, one can give a more direct proof of the first part as follows. Given a conformal connection $\nabla^{{[g]}},$ associated to $s\colon\cC_0\to\cC,$ one has \eqref{eq:rho-tensor-conformal} and \eqref{eq:Weyl-conformal-projective}, using which the projective Weyl curvature of $[\nabla^{[g]}]$ can be expressed in terms of $\sP,$ e.g.
  \begin{equation}\label{eq:proj-Weyl-rho}
    \begin{gathered}
          W^0_{112} = \ve\sP_{02},\quad  W^0_{212} = -\ve\sP_{01},\quad  W^1_{002} = \ve\sP_{12},\quad 
          W^1_{202} = -\sP_{10},\\
            W^1_{212} = \half(\ve\sP_{00}-\sP_{11}),\quad  W^2_{001} = \ve\sP_{21},\quad  W^2_{101} = -\sP_{20},\quad  W^2_{121} = \half(\ve\sP_{00}-\sP_{22}),\\
    \end{gathered}
  \end{equation}
where $\ve=1$ or $-1$ in the Riemannian or Lorentzian signature, respectively.    Projective flatness is equivalent to $W^i_{jkl}=0,$ which by \eqref{eq:proj-Weyl-rho} implies
  \begin{equation}\label{eq:rho-constant}
  \sP_{ij}=\sP_{00}\ve_{ij}.  
  \end{equation}
By Remark \ref{rmk:closed-weyl-structure} one obtains that  $\nabla^{[g]}$ is locally the Levi-Civita connection of a metric  $g_0\in [g],$ defined up to positive homotheties,  whose Schouten tensor satisfies \eqref{eq:rho-constant} and, therefore, by Schur's lemma,  $\sP_{00}$ is a constant defined up to positive homotheties, i.e. $g_0$ is Einstein.  Moreover,  the entries of the Cotton-York tensor, denoted as $Y_{ijk}$ in \eqref{eq:cotton-york}, are determined by the first jet of the Rho tensor $\sP_{ij}$ for any conformal connection via the relation
  \begin{equation}\label{eq:cotton-york-rho-tensor}
    Y_{ijk}=\sP_{ik;j}-\sP_{ij;k}.
  \end{equation}
  As a result, \eqref{eq:rho-constant} implies  $Y_{ijk}=0,$ i.e. the conformal structure is flat.


\end{proof}

The conformal geometry in dimension three is  special since the Weyl tensor for conformal structures in this dimension  is zero and the fundamental invariant is the Cotton-York tensor which has higher homogeneity. Consequently, the Cotton-York tensor is  determined by the Rho tensor, as given in \eqref{eq:cotton-york-rho-tensor}. Moreover, it follows  from \eqref{eq:proj-Weyl-rho} that  the projective Weyl curvature determines the trace-free part of the Rho tensor of the conformal connection.

Inspecting  Corollary \ref{cor:conf-rigid-flat} in  higher dimensions gives the following. 

\begin{theorem}\label{thm:conf-beltrami}
  A  conformal structure in dimension $\geq 3$  locally has a projectively flat conformal connection  if and only if it is locally conformally  flat.  
\end{theorem}
\begin{proof}
 Using Corollary \ref{cor:conf-rigid-flat}, we only need to prove the theorem in dimensions $\geq 4.$ 
  In this proof, unlike the rest of the paper, the range of indices is
  \[0\leq i,j,k,l,m\leq n-1,\]
  and $[\ve_{ij}]$ is the standard diagonal matrix of signature $(p,q)$ and $p+q=n.$ By our discussion in \ref{sec:conf-geom-dimens},   a (pseudo-)conformal structure of signature $(p,q)$ on an $n$-dimensional manifold $M$ corresponds to a  Cartan geometry  $(\cC\to M,\psi)$ of type $(\mathrm{SO}(p+1,q+1),Q_1).$   The structure equations on $\cC$ is known to be, e.g. see \cite[Chapter 4]{Kobayashi},
  \begin{equation}\label{eq:conformal-str-dim-n}
    \begin{aligned}
      \exd\omega^i=&-\theta^i_j\w\omega^j,\\      \exd\theta^i_j=&-\theta^i_k\w\theta^k_j-\omega^i\w\xi_j-\ve^{ki}\ve_{jl}\xi_k\w\omega^l+\delta^i_j\xi_k\w\omega^k+ \half C^i_{jkl}\omega^k\w\omega^l,\\
      \exd\xi_i=&\phantom{-}\theta^j_i\w\xi_j+\half C_{ijk}\omega^j\w\omega^k.
    \end{aligned}
  \end{equation}
   The conformal structure on $M$ is $[t^*g]\subset \mathrm{Sym}^2(T^*M),$ for any section $t\colon M\to\cC$, where  $g=\ve_{ij}\omega^i\omega^j\in \Gamma(\mathrm{Sym^2}(T^*\cC))$   Furthermore, $[\theta^i_j]$ is  $\mathfrak{co}(p,q)$-valued and the Weyl curvature entries $C^i_{jkl}$ satisfy
  \[C^i_{jkl}=-C^i_{jlk},\quad C^i_{[jkl]}=0,\quad C^i_{jil}=0.\]
  A Weyl structure $\nabla^{[g]}$ is equivalent to a section $s\colon\cC_0\to\cC,$ where $\cC_0\to M$ is a principal $\mathrm{CO}(p,q)$-bundle, using which one has 
  \[s^*\xi_i=\sP_{ij}\omega^j.\]
  Similarly, by our discussion in \ref{sec:proj-struct-dimens}, a projective structure $[\nabla]$ on an $n$-dimensional manifold is a Cartan geometry $(\cP\to M,\varphi)$ of type $(\mathrm{SL}(n+1,\RR),P_1)$ for which the structure equations    are, e.g. see \cite[Chapter 4]{Kobayashi},
  \begin{equation}\label{eq:projective-str-dim-n}
    \begin{aligned}
      \exd\tau^i=&-\omega^i_j\w\tau^j,\\      \exd\omega^i_j=&-\omega^i_k\w\omega^k_j-\tau^i\w\mu_j+\delta^i_j\mu_k\w\tau^k+ \half W^i_{jkl}\tau^k\w\tau^l,\\
      \exd\mu_i=&\phantom{-}\omega^j_i\w\mu_j+\half W_{ijk}\tau^j\w\tau^k,
    \end{aligned}
  \end{equation}
where $[\omega^i_j]$ is $\mathfrak{gl}(n)$-valued.  We point out that by comparing the expression for $\exd\tau^i$  using the matrix form \eqref{eq:3Dproj-CartanCurv} and using  \eqref{eq:projective-str-dim-n}, it follows that  the 1-forms $\omega^i_j$ above and $\eta^i_j$ in \eqref{eq:3Dproj-CartanCurv} are related by $\omega^i_j=\eta^i_j+\delta^i_j\eta^k_k.$ The entries of the projective Weyl curvature $W^i_{jkl}$ satisfy
\[W^i_{jkl}=-W^i_{jlk},\quad W^i_{[jkl]}=0,\quad W^i_{jik}=0.\]
Viewing the conformal connection $\nabla^{[g]}$ as a projective connection for $[\nabla^{[g]}],$ it defines a section $u\colon\cP_0\to \cP,$ where
\[u^*\mu_i=\sQ_{ij}\omega^j.\]
The objective is to express projective invariants  $W^i_{ijk}$ and $\sQ_{ij}$ in terms of conformal invariants $C^i_{jkl}$ and $\sP_{ij}.$ Firstly, note that using the natural inclusion $\iota\colon \cC_0\to\cP_0,$ one has
\[\iota^*\tau^i=\omega^i,\quad \iota^*\omega^i_{j}=\theta^i_j.\]
Consequently, after their pull-back to $\cC_0$, it follows from equating the right hand side of  the second structure equation in \eqref{eq:conformal-str-dim-n} and \eqref{eq:projective-str-dim-n}  that
\begin{equation}\label{eq:C-W-Q-P}  (C_{ijkl}+\ve_{il}\sP_{jk}-\ve_{jl}\sP_{ik}+\ve_{ij}\sP_{lk})\omega^k\w\omega^l=(W_{ijkl}-\ve_{ik}\sQ_{jl}+ \ve_{ij}\sQ_{lk})\omega^k\w\omega^l,
  \end{equation}
where we have used the conformal structure $[g]$ to lower indices, i.e.
\[C_{ijkl}=\ve_{im}C^m_{jkl},\quad W_{ijkl}=\ve_{im}W^m_{jkl}.\]
Since  $W^i_{jil}=C^i_{jil}=0,$ contracting  \eqref{eq:C-W-Q-P} with $\ve^{ik}$ and denoting $\sP=\ve^{ij}\sP_{ij},$ one obtains
\[\sP_{lj}-(n-1)\sP_{jl}-\sP\ve_{jl}=\sQ_{lj}-n\sQ_{jl}\Rightarrow \sQ_{ij}=\tfrac{1}{n-1}\sP\ve_{ij}+\tfrac{n^2-n-1}{n^2-1}\sP_{ij}-\tfrac{1}{n^2-1}\sP_{ji}, \]
which coincides with \eqref{eq:Q-P-3} when $n=3.$ Using the relation above and \eqref{eq:C-W-Q-P}, one can express  $W_{ijkl}$ in terms of $C_{ijkl}$ and $\sP_{ij}$ and arrive at
\begin{equation}\label{eq:W-C-P}
 W_{ik}:=\ve^{jl}W_{(ij)kl}=\tfrac{2}{n+1}\sP_{[ik]}+\tfrac{n-2}{n-1}(\sP\ve_{ik}-n\sP_{(ik)})
  \end{equation}
Since $W_{ij}$ is trace-free, it follows from \eqref{eq:W-C-P} that $W_{ij}=0$ implies $\sP_{ij}=c\ve_{ij}$ for some $c,$ which, by Remark \ref{rmk:closed-weyl-structure}, means $\nabla^{[g]}$ is the Levi-Civita connection of a metric  $g\in [g],$ defined up to positive homotheties, whose Schouten tensor satisfies $\sP_{ij}=cg_{ij},$ and, therefore, by Schur's lemma, $c$ is a constant defined up to positive homotheties, i.e. $g$ is Einstein.    Since projective flatness is equivalent to $W_{ijkl}=0,$ one has $\sP_{ij}=c\ve_{ij}$ and, consequently, the relation \eqref{eq:C-W-Q-P} gives $C_{ijkl}=0,$ i.e. $[g]$ is a flat conformal structure. 
As a result, if  $[\nabla^{[g]}]$ is projectively flat, then $\nabla^{[g]}$ is the Levi-Civita connection of (pseudo-)Riemannian metrics of constant sectional curvature.

\end{proof}
\begin{remark}\label{rmk:equivalent-description}
  Equivalently, Theorem \ref{thm:conf-beltrami} states that a projectively flat conformal connection in dimension $\geq 3$  is locally equivalent to the Levi-Civita  connection of the homothety class of (pseudo-)Riemannian metrics of constant sectional curvature. This is due to the fact that in a flat  conformal structure the Levi-Civita connections of the homothety class of (pseudo-)Riemannian metrics of constant sectional curvature are equivalent and projectively flat. Therefore, using Lemma \ref{lem:proj-conf-intersect}, the theorem states that, locally in any dimension $\geq 3$, this Levi-Civita connection is the unique projectively flat conformal connection.   I am grateful to Daniel J. F. Fox who pointed out to me an alternative proof of Theorem \ref{thm:conf-beltrami}  using \cite[Lemma 3.6]{Fox} by requiring the  projectively flat  AH structures to  be self-conjugate and noting that self-conjugate AH structures are equivalent to  Weyl structures.
  \end{remark}

  \begin{remark}  \label{rmk:2D-weyl}
Although projectively flat Weyl structures are rigid in dimension $\geq 3,$ one can show, e.g. using Cartan-K\"ahler theory, that on surfaces  they locally depend on 2 functions of 1 variable. A description of  conformal connections whose geodesics are great circles on 2-sphere can be found in \cite{Mettler-1, Mettler-2}.
\end{remark}

\subsection{The Einstein-Weyl condition}   
\label{sec:einst-weyl-repr}
Recall that the Ricci curvature of a linear connection  $\nabla$ on $M$, denoted by $\mathrm{Ric}^{\nabla},$ is defined by contracting its  curvature tensor and may not be symmetric. A conformal connection $\nabla^{[g]}$  is called Einstein-Weyl if the symmetric part of its Ricci tensor, $\mathrm{Sym}(\mathrm{Ric}^{\nabla^{[g]}}),$  satisfies
\[\mathrm{Sym}(\mathrm{Ric}^{\nabla^{[g]}})=f g,\]
for some function $f$ on $M$ and  $g\in [g].$

By our discussion in \ref{sec:proj-class-weyl}, any conformal connection  corresponds to a  principal $\mathrm{CO}(p,q)$-bundle $\cC_0\to M.$ Consider the associated $\mathrm{CO}(2,\RR)$-bundle $\what\mu\colon\cC_0\to\cT.$ Using the notational conventions in  \ref{sec:conventions}, the canonical coframe on  $\cC_0$ given by $(\omega^0,\omega^1,\omega^2,\theta^0,\cdots,\theta^3)$ allows one to define a frame on $\cC_0$ denoted by $(\partial_{\omega^0},\cdots,\partial_{\theta^3}).$   We have the following lemma for the line field $\ell=\langle\what\mu_*\partial_{\omega^0}\rangle\subset T\cT.$

\begin{lemma}\label{lem:path-geometry}
Given a conformal connection $\nabla^{[g]},$   the integral curves of the line field $\ell=\langle\what\mu_*\partial_{\omega^0}\rangle$ on $\cT$ are in one-to-one correspondence with (time-like) geodesics of $\nabla^{[g]}$ on $M.$
\end{lemma}
\begin{proof}
Using the discussion in \ref{sec:proj-class-weyl},  recall that a conformal connection $\nabla^{[g]}$ defines the projective connection \eqref{eq:Weyl-conformal-projective}. By structure equations, the line field $\ell$ and the integrable rank 2 distribution $\scV=\langle\what\mu_*\partial_{\theta^1},\what\mu_*\partial_{\theta^2}\rangle=\what\mu_*\Ker\{\omega^0,\omega^1,\omega^2\}\subset T\cT$ satisfy the bracket generating relation $[\ell,\scV]=T\cT,$  or, equivalently, 
\[\exd\omega^i\equiv \ve\theta^i\w\omega^0\mod\{\omega^1,\omega^2\},\]
which defines the multi-contact structure on $\PP TM$ mentioned in our discussion in \ref{sec:proj-struct-dimens} on the relation between geodesics of a projective structure and the integral curves of its geodesic spray,  $\ell.$
In the correspondence space theory of projective structures, e.g. see \cite[Section 4.4.3]{CS-Parabolic},  $\cT$ is said to have  a so-called \emph{path geometry} where the paths are the integral curves of $\ell$ and project to the geodesics of $\nabla^{[g]}.$  If $[g]$ is Lorentzian then, by \eqref{eq:conformal-metric}, the integral curves of $\ell$ are the unique lift of time-like geodesics of $\nabla^{[g]}$ to $\cT.$ 
\end{proof}
Now we have the following CR characterization of the Einstein-Weyl condition.
\begin{proposition}
  A conformal connection, $\nabla^{[g]},$ is Einstein-Weyl if and only if, in terms of the coframing  \eqref{eq:pull-back-CR-submanifold-cT}, the complex Pfaffian system $\{\zeta^1,\zeta^2\}$ is integrable on $\cT$. Subsequently, the   CR complex structure on the contact distribution  $\scC\subset T\cT$ descends to the 4-dimensional space of (time-like) geodesics of $\nabla^{[g]},$  also referred to as its minitwistor space. 
\end{proposition}
\begin{proof}
  By Lemma \ref{lem:path-geometry}, the space of (time-like) geodesics of $\nabla^{[g]},$ denoted by $\cS,$ coincides with the space of integral fields of $\what\mu_*\langle\partial_{\omega^0}\rangle\subset \cT.$ Thus, the set of 1-forms $(s^*\omega^1,s^*\omega^2,s^*\theta^1,s^*\theta^2)$, for any section $s\colon\cS\to \cC_0,$ gives an adapted coframe on $\cS.$ It is clear that if the complex Pfaffian system $\{\zeta^1,\zeta^2\}$ is integrable then the induced complex structure on $\cS$ with  holomorphic distribution $\Ker\{\zeta^1,\zeta^2\}\subset \CC\otimes T\cS$ is the descent of the CR holomorphic distribution given in the proof of Proposition \ref{prop:algebr-type-CR-twist} as $\Ker\{\omega^0,\zeta^1,\zeta^2\}.$

  It remains to prove the first part of the proposition. By our discussion in \ref{sec:proj-class-weyl}, given  a Weyl structure $s\colon\cC_0\to\cC$, it follows from the structure equations \eqref{eq:Cartan-curvature-conformal-conn} that the 1-forms $\zeta^1$ and $\zeta^2$ satisfy
  \begin{equation}\label{eq:EW-complex}
    \exd\zeta^1\equiv 0,\quad \exd\zeta^2\equiv -\half \left(\sP_{11}-\sP_{22}+\ri(\sP_{12}+\sP_{21})\right)\omega^0\w\czeta^1,
  \end{equation}
  modulo $\{\zeta^1,\zeta^2\}.$ %
  The  proposition follows if one shows that the vanishing conditions
  \begin{equation}\label{eq:EW-P}
        \sP_{11}-\sP_{22}=0,\quad  \sP_{12}+\sP_{21}=0,
      \end{equation}
      for a conformal connection implies the Einstein-Weyl condition. It is clear that if  $\nabla^{[g]}$ is Einstein-Weyl then \eqref{eq:EW-P} holds. Conversely, one needs to exploit the group action of $\mathrm{O}(p,q)$ on $\sP_{ij},$ which at the infinitesimal level are encoded in the following Bianchi identities:
      \begin{equation}\label{eq:bianchi1}
        \begin{aligned}
          \exd(\sP_{11}-\sP_{22})\equiv&  -\exd(\ri\sP_{21}+\ri\sP_{12})\equiv \half(\sP_{01}+\sP_{10}-\ri(\sP_{02}+\sP_{20}))\czeta^2,\\
          \exd(\sP_{01}+\sP_{10})\equiv& \half(\sP_{11}-\ve\sP_{00}-\ri\half (\sP_{12}+\sP_{21}))\czeta^2,\\
          \exd(\sP_{02}+\sP_{20})\equiv &  \half(\sP_{22}-\ve\sP_{00}+\ri(\sP_{12}+\sP_{21}))\czeta^2,
        \end{aligned}
      \end{equation}
      mod $\{\omega^0,\zeta^1,\zeta^2,\czeta^1,\theta^3,\theta^0\}.$ Solving  \eqref{eq:EW-P} and  its differential consequences  \eqref{eq:bianchi1}, one arrives at $\sP_{(ij)}=\sP_{00}\ve_{ij}.$ Since the symmetric parts of  $\sP_{ij}$ and the Ricci  tensor  $R_{ij}$ are related, as given in Remark \ref{rmk:closed-weyl-structure}, it follows that  $\nabla^{[g]}$ is Einstein-Weyl. 
\end{proof}

\begin{remark}  
The existence of a  complex structure on the minitwistor space of Einstein-Weyl 3-manifolds was  observed in \cite{Hitchin-EW}. In relation to  integrable $\mathrm{GL}(2,\RR)$-structures defined by scalar ODEs of odd order, a similar CR characterization via their twistor bundle is expected to hold. The twistor bundle of scalar ODEs of higher order  has been studied in  \cite{Krynski-GL2} when the order is even. 
\end{remark}

\subsection*{Acknowledgments}
I would like to thank  Daniel J. F.  Fox and David Sykes  for helpful conversations. I also thank the anonymous referee for many helpful and clarifying comments and corrections. I acknowledge partial support by the grant  PID2020-116126GB-I00 provided via the Spanish Ministerio de Ciencia e Innovaci\'on MCIN/ AEI /10.13039/50110001103 as well as partial funding from the Tromsø Research Foundation (project “Pure Mathematics in Norway”) and the UiT Aurora project MASCOT. Some parts of this article was done when I was at Masaryk University for which I would like to acknowledge the grant  GA22-00091S.  The  EDS calculations  are done using Jeanne Clelland's \texttt{Cartan} package in Maple. %

\bibliographystyle{alpha}      
\bibliography{Refs.bib}
\end{document}